\documentclass[a4paper,reqno,oneside,10pt]{amsart}

\usepackage{amsthm,amsmath,amssymb,stmaryrd,euscript,mathrsfs}
\usepackage{marginnote,cleveref}
\usepackage{amssymb,amsmath,amsthm}
\usepackage{mathtools,url}
\usepackage{stmaryrd}
\usepackage[shortlabels]{enumitem}
\usepackage[all,cmtip,2cell]{xy} 
\usepackage[utf8]{inputenc}
\usepackage{geometry}
\usepackage{cleveref,marginnote}

\crefname{section}{Section}{Sections}

\setlist[itemize]{leftmargin=2em}

\numberwithin{equation}{section}
\makeatletter

\renewcommand{\section}{\@startsection
{section}
{1}
{0mm}
{-\baselineskip}
{0.5\baselineskip}
{\large \bfseries}}

\renewcommand{\subsection}{\@startsection
{subsection}
{2}
{0mm}
{1mm}
{-1ex}
{\normalfont\normalsize\bfseries}}

\makeatother

\def\defthm#1#2#3#4{
  \newtheorem{#1}[theorem]{#3}
  \newtheorem*{#1*}{#3}
  \newtheorem{#2}[theorem]{#4}
  \newtheorem*{#2*}{#4}
  \crefname{#1}{#3}{#4}
  \crefname{#2}{#4}{#4}  
}

\newtheoremstyle{mythm}%
{10pt}
{}
{\itshape}
{}
{\bf}
{.}
{.5em}
{}%

\newtheoremstyle{mydef}%
{10pt}
{3pt}
{}
{}
{\bf}
{.}
{.5em}
{}%

\newtheoremstyle{myrmk}%
{10pt}
{3pt}
{}
{}
{\em}
{.}
{.5em}
{}%

\theoremstyle{mythm}
\newtheorem{theorem}{Theorem}[section]
\newtheorem*{theorem*}{Theorem}
\crefname{theorem}{Theorem}{Theorems}
\defthm{corollary}{corollaries}{Corollary}{Corollaries}
\defthm{lemma}{lemmata}{Lemma}{Lemmata}
\defthm{proposition}{propositions}{Proposition}{Propositions}
\theoremstyle{mydef}
\defthm{definition}{definitions}{Definition}{Definitions}
\defthm{example}{examples}{Example}{Examples}

\theoremstyle{myrmk}
\defthm{remark}{remarks}{Remark}{Remarks}
\defthm{convention}{Convention}{Convention}{Conventions}

\defthm{question}{questions}{Question}{Questions}

\newcommand{\ie}{{i.e.}}
\newcommand{\cf}{{cf.}}
\newcommand{\id}{\mathrm{Id}}

\newcommand{\myemph}{\textit}
\newcommand{\defeq}{=_{\mathrm{def}}}
\newcommand{\iso}{\cong}

\newcommand{\comp}{\,}
\newcommand{\kcomp}{\circ}

\newcommand{\op}{\mathrm{op}}

\newcommand{\co}{\colon}

\newcommand{\cat}{\mathbb}
\newcommand{\catX}{\cat{X}}
\newcommand{\catY}{\cat{Y}}
\newcommand{\catZ}{\cat{Z}}
\newcommand{\catA}{\cat{A}}
\newcommand{\catB}{\cat{B}}
\newcommand{\catC}{\cat{C}}

\newcommand{\cal}{\mathcal}
\newcommand{\psh}[1]{P{#1}}
\newcommand{\pshh}{P}
\newcommand{\pshA}{\psh{(\catA)}}
\newcommand{\pshB}{\psh{(\catB)}}
\newcommand{\pshX}{\psh{(\catX)}}
\newcommand{\pshY}{\psh{(\catY)}}
\newcommand{\pshZ}{\psh{(\catZ)}}
\newcommand{\yo}{\mathrm{y}}
\newcommand{\yon}{\mathrm{y}}

\newcommand{\ind}{D}
\newcommand{\indX}{\ind(\catX)}
\newcommand{\indY}{\ind(\catY)}

\newcommand{\shyA}{\pshA}
\newcommand{\shyB}{\pshB}

\newcommand{\bcat}{\cal}
\newcommand{\bC}{\bcat{C}}
\newcommand{\bD}{\bcat{D}}
\newcommand{\bE}{\bcat{E}}

\newcommand{\tcat}{\mathbf}
\newcommand{\Set}{\tcat{Set}}

\newcommand{\Cat}{\tcat{Cat}}
\newcommand{\CAT}{\tcat{CAT}}
\newcommand{\COC}{\tcat{COC}}
\newcommand{\FIL}{\tcat{FIL}}
\newcommand{\MONCOC}{\tcat{MONCOC}}
\newcommand{\Prof}{\tcat{Prof}}
\newcommand{\SProf}{\ms\text{-}\Prof}
\newcommand{\Mon}{\tcat{Mon}}
\newcommand{\MON}{\tcat{MON}}
\newcommand{\MonCoc}{\tcat{MONCOC}}

\newcommand{\Esp}{\mathbf{Esp}}
\newcommand{\CatSym}{\mathbf{CatSym}}

\newcommand{\ms}{S}
\newcommand{\es}{e}
\newcommand{\mus}{m}
\newcommand{\mss}{\ms^2}

\newcommand{\mt}{T}
\newcommand{\et}{i}
 \newcommand{\mut}{n}
 \newcommand{\mtt}{\mt^2}
 \newcommand{\mttt}{\mt^3}

 \newcommand{\meta}{\eta} 
 \newcommand{\mmu}{\mu}
 \newcommand{\mbeta}{\theta}
 \newcommand{\mepsilon}{\varepsilon}

 \newcommand{\mts}{\mt \ms}
 
 \newcommand{\mps}{\psh \ms}

\newcommand{\da}{\ast}
\newcommand{\sh}{\sharp}
\newcommand{\kl}{\mathrm{Kl}}

\newcommand{\Salg}{\ms\text{-}\mathrm{Alg}}
\newcommand{\psSalg}{\mathrm{Ps}\text{-}\ms\text{-}\mathrm{Alg}}
\newcommand{\Spalg}{\ms'\text{-}\mathrm{Alg}}

\newcommand{\liftT}{\bar{\mt }}

\newcommand{\extS}{\tilde{S}}

\newcommand{\exte}{\tilde{e}}

\newcommand{\fra}{((h^\da \comp g)^\da \eta_f)^\da}
\newcommand{\frb}{\varepsilon_{(h^\da g)^\da \comp f^\da}}
\newcommand{\frc}{(h^\da \comp \eta_g)^\da \comp f^\da}
\newcommand{\frd}{\varepsilon_{h^\da \comp g^\da} \comp f^\da}
\newcommand{\fre}{h^\da \comp \varepsilon_{g^\da \comp f^\da}}
\newcommand{\frf}{h^\da \comp (g^\da \comp \eta_f)^\da}
\newcommand{\frg}{\varepsilon_{h^\da \comp (g^\da \comp f)^\da}}
\newcommand{\frh}{(h^\da \comp \eta_{g^\da \comp f})^\da}
\newcommand{\fri}{(\varepsilon_{h^\da \comp g^\da} \comp f)^\da}
\newcommand{\frj}{((h^\da \comp \eta_g)^\da \comp f)^\da }
\newcommand{\frmm}{((h^\da \comp g^\da \comp \et)^\da \comp \eta_f)^\da}
\newcommand{\frn}{\varepsilon_{(h^\da \comp g^\da \comp \et)^\da \comp f^\da}}
\newcommand{\fro}{(\varepsilon_{h^\da \comp g^\da} \comp f^\da \comp \et)^\da}
\newcommand{\frp}{\varepsilon_{h^\da \comp g^\da \comp f^\da}}
\newcommand{\frq}{(h^\da \comp \eta_{g^\da \comp f^\da \comp \et})}
\newcommand{\frr}{\varepsilon_{h^\da ( g^\da \comp f^\da \comp \et)^\da}}
\newcommand{\frs}{(h^\da \comp (g^\da \comp \eta_f)^\da \comp \et)}
\newcommand{\frt}{(h^\da \comp g^\da  \eta_f)^\da}
\newcommand{\fru}{((h^\da \comp \eta_g)^\da \comp f^\da \comp \et)^\da}

\title[Relative pseudomonads]{Relative pseudomonads, Kleisli bicategories, \\ and substitution monoidal structures}
\author{M. Fiore}
\address{Computer Laboratory, University of Cambridge}
\email{Marcelo.Fiore@cl.cam.ac.uk}

\author{N. Gambino}
\address{School of Mathematics, University of Leeds} 
\email{n.gambino@leeds.ac.uk}

\author{M. Hyland}
\address{DPMMS, University of Cambridge}
\email{m.hyland@dpmms.cam.ac.uk}

\author{G. Winskel}
\address{Computer Laboratory, University of Cambridge}
\email{Glynn.Winskel@cl.cam.ac.uk}

\date{September 28th, 2017}

\begin{document}

\begin{abstract} We introduce the notion of a relative pseudomonad, which generalizes
the notion of a pseudomonad, and define the Kleisli bicategory associated to a relative 
pseudomonad. We then present an efficient method to define pseudomonads on
the Kleisli bicategory of a relative pseudomonad. The results are applied to define 
 several pseudomonads on the bicategory of profunctors in an homogeneous way
 and provide a uniform approach to the definition of bicategories that are of interest in operad theory, 
mathematical logic, and theoretical computer science.
\end{abstract}

\maketitle

\section{Introduction}

Just as classical monad theory provides a general approach to study algebraic structures on objects of a category (see~\cite{BarrM:toptt} for example), 2-dimensional monad theory offers an elegant way to investigate algebraic structures on objects of a 2-category~\cite{BlackwellR:twodmt,HylandM:psempc,KellyG:cohtla,KellyG:prolt,PowerJ:agencr,ShulmanM:notepe}.  
Even if the strict notion of a 2-monad
is sufficient to develop large parts of the theory,  the strictness requirements that are part of its definition are too restrictive for some applications and it is necessary to work with the notion of a pseudomonad~\cite{BungeM:cohera}, in which the diagrams expressing the associativity and unit axioms for a 2-monad commute up to specified invertible modifications, rather than strictly. In recent years, pseudomonads have been studied 
extensively~\cite{ChengE:psdist,LackS:cohapm,MarmolejoF:dislp,MarmolejoF:cohplr,MarmolejoF:kaneli,MarmolejoF:noip,TanakaM:psedl}.

Our general aim here is to develop further the theory of pseudomonads. In particular, we introduce 
relative pseudomonads, which generalize pseudomonads, 
define the associated Kleisli bicategory of a relative pseudomonad, and describe a method to extend a 2-monad on a \mbox{2-category} 
to a pseudomonad on the Kleisli bicategory of a relative pseudomonad.  We use this method to show
how several 2-monads on the \mbox{2-category} $\Cat$ of small categories and functors can be extended to pseudomonads on the 
bicategory $\Prof$ of small categories and profunctors~(also known as bimodules or 
distributors)~\cite{BenabouJ:disw,LawvereF:metsgl,StreetR:fibb}. This result has applications in the theory of variable binding~\cite{FioreM:secodsa,FioreM:abssvb,PowerJ:binsgc,TanakaM:abssvllb},  concurrency~\cite{CattaniG:proomb}, species of structures~\cite{FioreM:carcbg}, models of the differential $\lambda$-calculus~\cite{FioreM:matmcc}, and operads and multicategories~\cite{ChikhladzeD:laxftm,CrutwellG:unifgm,CurienPL:opecdl,GambinoN:opebaf,GurskiN:opetpc}.

For these applications, one would like to regard the bicategory of profunctors as a Kleisli bicategory and then use the theory of pseudo-distributive laws~\cite{KellyG:cohtla,MarmolejoF:dislp,MarmolejoF:cohplr}, 
\ie~the 2-dimensional counterpart of Beck's fundamental work on distributive laws~\cite{BeckJ:disl} (see~\cite{StreetR:fortm} for an abstract treatment). In order to carry out this idea, one is naturally led to try to consider the presheaf construction, which sends a small category $\catX$ to its category of presheaves $\pshX \defeq [\catX^\op, \Set]$, as a pseudomonad. Indeed, a profunctor $F \co \catX \to \catY$, i.e. a functor $F \co \catY^\op \times \catX \to \Set$, can be  identified with a functor $F \co \catX \to \pshY$. However, the presheaf construction fails to be a pseudomonad for size reasons, since it sends small categories to locally small ones, making it impossible to define a multiplication. Although some aspects of the theory can be developed restricting the attention to small presheaves~\cite{DayB:limsf}, which support the structure of a pseudomonad, some of our applications naturally involve  general presheaves and thus require us to deal not only with coherence but also size issues. 

In order to do so, we introduce the notion of a relative pseudomonad (\cref{thm:relative-pseudomonad}), which is based on the notions of a relative monad~\cite[Definition~2.1]{AltenkirchT:monnnb} and of a no-iteration pseudomonad~\cite[Definition~2.1]{MarmolejoF:noip}. These notions are, in turn, inspired by Manes' notion of a Kleisli triple~\cite{ManesE:algt}, which is equivalent to that of a monad, but better suited to define Kleisli categories~(see 
also~\cite{MarmolejoF:monesn,WaltersR:catau}). 
For a pseudofunctor between bicategories $J \co \bC \to \bD$ (which in our main example is the inclusion $J \co \Cat \to \CAT$ of the \mbox{2-category} of small categories into the \mbox{2-category} of locally small categories), the core of the data for a  relative pseudomonad~$T$ over $J$ consists of an object $TX \in \bD$ for every  $X \in \bC$, a morphism $\et_X \co JX \to \mt X$  for every $X \in \bC$, and a morphism $f^\da \co 
\mt X \to \mt Y$ for every $f \co JX \to \mt Y$ in $\bD$. This is  as in a relative monad, but the equations for a relative monad are replaced in a relative pseudomonad 
by families of invertible 2-cells satisfying appropriate coherence conditions, as in a no-iteration pseudomonad. As we will see in \cref{thm:kbicat}, these conditions imply that every relative pseudomonad $\mt$ over $J \co \bC \to \bD$ has an 
associated Kleisli bicategory~$\kl(\mt)$, defined analogously to the one-dimensional case. In our main example, the presheaf construction gives rise  to a relative pseudomonad over the inclusion $J \co \Cat \to \CAT$ in a natural way and it is then immediate to identify its Kleisli bicategory with the bicategory of profunctors. It should be noted here that the presheaf construction is
neither a no-iteration pseudomonad (because of size issues) nor a relative monad (because of strictness issues).

As part of our development of the theory of relative pseudomonads, we show how relative pseudomonads generalize 
no-iteration pseudomonads (\cref{thm:no-iteration-is-rel-id}) and hence (by the results in~\cite{MarmolejoF:noip}) also pseudomonads. We then introduce relative pseudoadjunctions, which are related to relative pseudomonads just as adjunctions are connected to monads. In particular, we show that every relative pseudoadjunction gives rise to a relative pseudomonad (\cref{thm:adjkls}) 
 and that the Kleisli bicategory associated to a relative pseudomonad fits in a relative pseudoadjunction~(\cref{thm:psekle}).  
 
 Furthermore, we introduce and study the notion of a lax idempotent relative pseudomonad, which appears to be the appropriate counterpart in our setting of the notion of a lax idempotent 2-monad (often called Kock-Z\"oberlein doctrines)~\cite{KellyG:prolt,KockA:monwsa,ZoberleinV:dokat} and pseudomonad~\cite{MarmolejoF:docwsf,MarmolejoF:kaneli,StreetR:fibb}.
In~\cref{prop:laxid} we will give several equivalent characterizations of lax idempotent relative pseudomonad and combine this result with the analogous one in~\cite{MarmolejoF:kaneli} to show that a 
pseudomonad  is lax idempotent as a pseudomonad in the usual sense only if it is lax idempotent as a relative pseudomonad in our sense. 
This notion is of  interest since it allows us to exhibit examples of relative pseudomonads by reducing the verification of the 
coherence axioms for a relative pseudomonad to the verification of certain universal properties.  In particular, the relative pseudomonad of presheaves can be constructed in this way.  

We then consider the question of when a 2-monad on the 2-category $\Cat$ of small categories can be extended to a pseudomonad on the bicategory $\Prof$ of profunctors. Rather than adapting  the theory of distributive laws to relative pseudomonads along the lines of what has been done for no-iteration monads~\cite{MarmolejoF:monesn}, which would involve complex calculations with coherence conditions, we establish directly that,  for a pseudofunctor $J \co \bC \to \bD$ of 2-categories, a 2-monad~$S \co \bD \to \bD$ restricting to~$\bC$ along $J$ in
a suitable way, and a relative pseudomonad $T$ over $J$, if $T$ admits a lifting to 2-categories of  strict algebras or pseudoalgebras for~$S$, then~$S$ admits an extension to the Kleisli bicategory 
of~$T$~(\cref{thm:main}). We do so bypassing the notion of a pseudodistributive law in a counterpart of Beck's result. 

This result is well-suited to our applications, where the structure that manifests itself most naturally is that of a lift of the relative pseudomonad of presheaves to various
2-categories of categories equipped with algebraic structure, often via forms of Day's convolution monoidal structure~\cite{DayB:clocf,ImG:unipcm}.  In particular, our results will imply that the 2-monads for several important notions (categories with terminal object, categories with  finite products, categories with finite limits, monoidal categories, symmetric monoidal categories, unbiased monoidal categories, unbiased symmetric  mo\-noi\-dal categories, strict monoidal categories, and symmetric strict monoidal categories) can be extended to pseudomonads on the bicategory of profunctors. A reason for interest in this result is that the compositions in the Kleisli bicategories of these pseudomonads can be understood as variants of the substitution monoidal structure that can be used to characterize
notions of operad~\cite{KellyG:onojpm,SmirnovV:onccts}.

As an illustration of the applications of our theory, we discuss our results in the special case of the 2-monad $S$ for symmetric strict monoidal categories, showing how it can be extended to a pseudomonad on the bicategory of profunctors. This result is  the cornerstone of the understanding of the bicategory of generalized species of structures defined in~\cite{FioreM:carcbg} as a `categorified' version of the relational model of linear logic~\cite{HylandM:engeler,HylandM:somrgd} and leads to showing that the substitution monoidal structure that gives rise to the notion of a coloured  operad~\cite{BaezJ:higda} is a special case of the composition in the Kleisli bicategory. The results presented here are intended to make these ideas precise by dealing with both size and coherence issues in a conceptually clear way.

\subsection*{Organization of the paper} \cref{sec:background} reviews some background material on 2-monads, pseudomonads and their algebras. Our development
starts in \cref{sec:relative-pseudomonads}, where we introduce relative pseudomonads, relate them to
no-iteration pseudomonads and ordinary pseudomonads, introduce relative pseudoadjunctions and establish a connection between relative pseudoadjunctions and 
relative pseudomonads.
\cref{sec:kleb} defines the Kleisli bicategory associated to a relative pseudomonad and discusses some of its basic properties.  In~\cref{sec:lax-idempotency} we introduce and study lax idempotent relative pseudomonads. \cref{sec:liftings-extensions-compositions} shows that 
an extension of a relative pseudomonad~$T$ to 2-categories of strict algebras or pseudoalgebras for a 2-monad~$S$ induces an extension of~$S$ to the Kleisli bicategory of~$T$,
as well as a composite relative pseudomonad $\mts$. 
We conclude the paper in \cref{sec:applications} by discussing applications of our theory and showing how several 2-monads on $\Cat$ can be extended  to pseudomonads on $\Prof$.

\subsection*{Acknowledgements} Nicola Gambino acknowledges gratefully the support of EPSRC, under grant~EP/M01729X/1, and of the US Air Force Research Laboratory, 
under agreement number FA8655-13-1-3038. We are grateful to the anonymous referee for insightful comments, which led in particular to a simplification of the material 
in~\cref{sec:lax-idempotency}. 

\section{Background}
\label{sec:background}

\subsection*{2-categories and 2-monads} We assume that readers are familiar with the fundamental aspects of the theory of 2-categories and of bicategories (as presented, for
example, in~\cite{BenabouJ:intb,BorceuxF:hancaI,LackS:a2cc}) and confine ourselves to review some facts that will be used in the following and to fix notation and conventions. 

For a 2-category $\bC$ and a pair of objects $X, Y \in \bC$, we write $\bC[X,Y]$ for the hom-category of morphisms $f \co X \to Y$ and 2-cells between them, which we denote with lower-case Greek letters, $\phi \co f \to f'$.  Two parallel morphisms $f \, , f' \co X \to Y$
are said to be \emph{isomorphic} if they are isomorphic as objects of $\bC[X,Y]$, and we write $f \iso f'$ in this case. We write $\CAT$ 
 for the 2-category of locally small categories, functors, and natural transformations. Its full sub-2-category spanned by small categories will be written $\Cat$. We then have an inclusion $J \co \Cat \to \CAT$. We use the terms \emph{pseudofunctor}, \emph{pseudonatural transformation}, and \emph{pseudoadjunction} rather than homomorphism,  strong natural transformation, and
 biadjunction, respectively. 
 
Let us now review some aspects of 2-dimensional monad theory~\cite{BlackwellR:twodmt}. By a \emph{$2$-monad} on a 2-category $\bC$ we mean a 2-functor $\ms \co \bC \to \bC$ equipped with 2-natural transformations $m \co S^2 \to S$ and 
$e \co 1_\bC \to S$, called the \emph{multiplication} and \emph{unit} of the 2-monad, respectively, satisfying the usual axioms for a monad in a strict sense.  As usual, we often refer to a 2-monad by mentioning only its underlying 2-functor, leaving implicit the rest of its data. Similar conventions will be used for other kinds of structures considered in the rest of the paper. 

For a 2-category $\bC$ and 2-monad~$S \co \bC \to \bC$, we write  $\psSalg_\bC$ (or $\psSalg$ if no confusion arises) for the 2-category of pseudoalgebras, pseudomorphisms and algebra 2-cells, and~$\Salg_\bC$ (or~$\Salg$) for the locally full sub-2-category of $\psSalg_\bC$ spanned by strict algebras. Here, by a 
\emph{pseudoalgebra} we mean an object $A \in \bC$, called the \emph{underlying object} of the algebra,
equipped with a 
morphism $a \co \ms A \to A$, called the \emph{structure map} of the algebra, and invertible  $2$-cells 
\[
\begin{xy}
(-12,8)*+{\mss A}="4";
(-12,-8)*+{\ms A }="3";
(12,8)*+{\ms A}="2";
(12,-8)*+{A \, ,}="1"; 
(-1,2)="5";
(-1,-3)="6";
{\ar^{\ms a} "4";"2"};
{\ar_{a} "3";"1"};
{\ar_{m_A} "4";"3"};
{\ar^{a} "2";"1"};
{\ar@{=>}^{\ \scriptstyle{\bar{a}}} "5";"6"};
\end{xy}
\hspace{1.5cm}
\begin{xy}
(-8,8)*+{A}="4";
(8,8)*+{\ms A}="2";
(8,-8)*+{A\, ,}="1";
(5,1)="5";
(0,1)="6";
{\ar@/_/_{1_A} "4";"1"};
{\ar^(.45){e_A} "4";"2"};
{\ar^{a} "2";"1"};
{\ar@{=>}^(.4){{\scriptstyle{\tilde{a}}}} "6";"5"};
\end{xy} 
\]
called the \emph{associativity} and \emph{unit} 2-cells of the algebra, subject to two coherence axioms~\cite{StreetR:fibyl}.  We have a \emph{strict algebra} when the associativity and unit 2-cells are identities, in which case (as in analogous cases below) the
coherence conditions are satisfied trivially. For pseudoalgebras~$A$ and~$B$ (and in particular for strict algebras), a pseudomorphism from $A$ to $B$ consists of a morphism $f \co A \rightarrow B$ and an invertible $2$-cell 
\begin{equation}
\label{equ:marcelo}
\begin{gathered}
\begin{xy}
(-12,8)*+{\ms A}="4";
(-12,-8)*+{ A }="3";
(12,8)*+{\ms B}="2";
(12,-8)*+{B \, ,}="1"; 
(-1,2)="5";
(-1,-3)="6";
{\ar^{\ms f} "4";"2"};
{\ar_{f} "3";"1"};
{\ar_{a} "4";"3"};
{\ar^{b} "2";"1"};
{\ar@{=>}^{\ \bar{f}}"5";"6"};
\end{xy}
\end{gathered}
\end{equation}
required to satisfy two coherence axioms~\cite{BlackwellR:twodmt,StreetR:fibyl}. For pseudomorphisms $f, g \co A \to B$, an \emph{algebra $2$-cell} between them is a  $2$-cell $\alpha: f \rightarrow g$ that satisfies one coherence  axiom~\cite{BlackwellR:twodmt}. We have  a forgetful 2-functor $U \co \psSalg \to \bC$ with a left pseudoadjoint
$F \co \bC \to \psSalg$, defined by mapping 
an object $X \in \bC$ to the \emph{free algebra} on $X$, which is the strict algebra having $\ms X$ as its underlying object and $\mus_X \co \mss X \to \ms X$ as its structure map. The components of the unit of the pseudoadjunction are the components of the unit of the 2-monad.

In our applications, we will consider several 2-monads on $\CAT$ (restricting to $\Cat$ in an evident way), for which we invite the readers to consult~\cite{BlackwellR:twodmt,LawvereW:ordsed,LeinsterT:higohc}.   Among them, the 2-monads for (strict) monoidal categories, symmetric (strict) monoidal categories, 
categories with finite limits, categories with finite products, and categories with a terminal object.

\subsection*{Bicategories and pseudomonads} For a bicategory  $\bC$, we 
write the associativity and unit isomorphisms as natural families of invertible 2-cells 
\begin{equation}
\label{equ:bicat-isos}
\xymatrix{ (h \comp g) \comp f \ar[r]^-{\iso} & h \comp (g \comp f)}  \, , \quad
\xymatrix{ 1_Y  f \ar[r]^-{\iso} & f}  \, , \quad 
\xymatrix{ f \ar[r]^-{\iso} & f \comp 1_X \, .}
\end{equation}
which we leave unnamed. By the coherence theorem for bicategories~\cite{MacLaneS:coh-bicategories} (which also follows from the
bicategorical Yoneda lemma~\cite{StreetR:fibb}, see~\cite{GordonR:cohb}), every bicategory is biequivalent to a 2-category.
In virtue of this, we shall often treat bicategories as if they were 2-categories.

\begin{example} \label{exa:prof}
Fundamental to our applications is the bicategory $\Prof$ of profunctors~\cite{BenabouJ:disw,LawvereF:metsgl,StreetR:fibb}. Its objects are
small categories; and for small categories $\catX$ and $\catY$, the hom-category $\Prof[\catX, \catY]$ is defined to be $\CAT[\catY^\op \times \catX, \Set]$. 
The composite of profunctors $F \co \catX \to \catY$ and $G \co \catY \to \catZ$ is given by the profunctor $G \circ F \co \catX \to \catZ$ defined by the coend formula
\begin{equation} 
\label{equ:prof-comp}
(G\circ F)(z,x) \defeq \int^{y \in \catY} G(z,y) \times F(y,x) \, .
\end{equation} 
For a small category $\catX$, the identity profunctor $\id_\catX \co \catX \to \catX$ is defined by letting 
\begin{equation}
\label{equ:prof-id}
 \id_\catX(x,y) \defeq \catX[x,y] \, . 
\end{equation}
It remains to prove that these definitions give rise to a bicategory.
In the literature, it is often suggested that this can be proved by
direct calculations, left to the readers. In \cref{sec:relative-pseudomonads}, instead, we will give a  more conceptual proof by describing $\Prof$ as the Kleisli bicategory associated to the relative pseudomonad of presheaves. 
\end{example}

 A \emph{pseudomonad} on a bicategory $\bC$ is given by a 
pseudofunctor $T \co \bC  \rightarrow \bC$, pseudonatural transformations $\mut \co T^2 \rightarrow T$ and $i \co  1_{\bC} \rightarrow T$, called
the \emph{multiplication} and \emph{unit} of the pseudomonad, respectively, and invertible modifications $\alpha$, $\rho$, and $\lambda$ ,  
called the  \emph{associativity}, \emph{right unit}, and \emph{left unit}, respectively, of $T$, fitting in the diagrams 
\begin{equation}
\label{equ:pseudomonad-modifications}
\begin{gathered}
\begin{xy}
(0,16)*+{\mttt}="4";
(24,16)*+{\mtt}="3";
(0,0)*+{\mtt}="2";
(24,0)*+{\mt \, , }="1"; 
(12,10)="5";
(12,5)="6";
{\ar_{\mut \mt} "4";"2"};
{\ar^{\mut} "3";"1"};
{\ar^{\mt \mut} "4";"3"};
{\ar_{\mut} "2";"1"};
 {\ar@{=>}^(.35){\ \scriptstyle{\alpha}} "5";"6"};
\end{xy} \qquad 
\begin{xy}
(0,16)*+{\mt}="4";
(16,16)*+{\mtt}="2";
(16,0)*+{\mt  }="1";
(32,16)*+{\mt}="3";
(8,10)="5";
(12,10)="6";
(21,10)="7";
(25,10)="8";
{\ar@/_/_{1} "4";"1"};
{\ar^(.45){\et \mt} "4";"2"};
{\ar^{\mut} "2";"1"};
{\ar@/^/^{1} "3";"1"};
{\ar_{\mt \et} "3";"2"};
{\ar@{=>}^(.4){{\scriptstyle{\lambda}}\; } "5";"6"};
{\ar@{=>}^(.3){{\scriptstyle{\rho}}\; } "7";"8"};
\end{xy} 
\end{gathered}
\end{equation}
and subject to two coherence conditions~\cite{LackS:cohapm}. The notions of a strict algebra and pseudoalgebra, of
strict morphism and pseudomorphism, and of algebra 2-cell make sense also for pseudomonads, giving rise to bicategories $\psSalg$ and $\Salg$. When $\bC$ is a 2-category, these are again 2-categories. 

Every pseudomonad  has also an associated Kleisli bicategory~\cite{ChengE:psdist}, which can be defined in
complete analogy with the one-dimensional case; but we do not spell this out since we will give an alternative account of
the Kleisli construction in \cref{sec:relative-pseudomonads}. Importantly, in constrast with the situation for algebras discussed above,
the Kleisli construction for a pseudomonad $T \co \bC \to \bC$ produces only a bicategory 
even when $\bC$ is a $2$-category, with the associativity and unit isomorphisms of $T$ used to give the associativity and unit isomorphisms
of the Kleisli bicategory~(see also~\cref{thm:kbicat} below).

\section{Relative pseudomonads} 
\label{sec:relative-pseudomonads}

In ordinary category theory, the notion of a monad has an equivalent alternative presentation, via the notion of a Kleisli
triple~\cite{ManesE:algt}, which is particularly convenient to define Kleisli categories. The notion of a Kleisli triple admits a natural generalization, given by the notion of a relative monad~\cite{AltenkirchT:monnnb}, which is obtained by allowing 
the underlying mapping on objects of the Kleisli triple to be defined  relative to a fixed functor (see~\cite{AltenkirchT:monnnb} for details). 
Similarly, in 2-dimensional category theory, the notion of a pseudomonad can be rephrased equivalently as the notion of a no-iteration 
pseudomonad~\cite{MarmolejoF:noip}, which is the 2-dimensional analogue of the notion of a Kleisli triple. 
Here, we introduce relative pseudomonads, which generalize no-iteration pseudomonads 
in the same way as relative monads generalize Kleisli triples, 
\ie~by allowing the mapping on objects that is part of a no-iteration pseudomonad to be defined 
relatively to a fixed pseudofunctor between bicategories. From now until the end of this section, we consider a fixed pseudofunctor between bicategories $J \co \bC \to \bD$.

\begin{definition} \label{thm:relative-pseudomonad}   
A \myemph{relative pseudomonad} $T$ over $J    \co \bC \to \bD$ consists of 
\begin{itemize} 
\item an object $\mt X \in \bD$, for every $X \in \bC$,
\item a family of functors $(-)^\da_{X,Y}   \co \bD[JX, \mt  Y]  \rightarrow   \bD[\mt  X, \mt  Y]$ 
for $X, Y \in \bC$,
\item a family of morphisms $\et_X \co JX \rightarrow \mt   X$ in $\bD$ for $X \in \bC$,
\item a natural family of invertible 2-cells $\mmu_{g, f} \co (g^\da \comp f)^\da \rightarrow g^\da \comp f^\da$, 
for $f \co JX \rightarrow \mt  Y$, $g\co JY \rightarrow \mt  Z$,
\item a natural family of invertible 2-cells $\meta_f \co f \rightarrow f^\da \comp \et_X$, for $f \co JX \to \mt Y$ in $\bD$,
\item a family of invertible 2-cells $\mbeta_X \co {\et_X}^\da \rightarrow 1_{\mt  X}$, for $X \in \bC$,
\end{itemize} 
such that the following conditions hold: 
\begin{itemize}   
\item for every $f \co JX \to \mt Y$, $g \co JY \to \mt Z$, $h \co J Z \to T V$, the diagram
\begin{equation} 
\label{equ:assoc-relativepseudomonad} 
\begin{gathered} 
\xymatrix{
 & ((h^\da \comp g)^\da f)^\da \ar[dr]^{\mmu_{ h^\da g, f}} \ar[dl]_{(\mmu_{h, g} \comp f)^\da} &  \\
 ( (h^\da g^\da) f )^\da \ar[d]_{\iso}  & & (h^\da g)^\da f^\da \ar[d]^{\mmu_{h, g} \comp f^\da} \\
  ( h^\da (g^\da f) )^\da  \ar[d]_{\mmu_{h, g^\da f}} & & (h^\da g^\da) f^\da \ar[d]^{\iso} \\ 
 h^\da (g^\da f)^\da \ar[rr]_{h^\da \comp \mmu_{g, f}} & & h^\da (g^\da f^\da) }
 \end{gathered}
 \end{equation} 
 commutes, and
 \item for every $f \co J X \to \mt Y$, the diagram
 \begin{equation}
 \label{equ:unit-relativepseudomonad} 
 \begin{gathered} 
\xymatrix@C+3ex{
f^\da \ar[r]^-{(\meta_f)^\da} \ar@/_1pc/[drr]_{\iso} &  (f^\da \comp \et_X)^\da \ar[r]^-{\mmu_{f, \et_X}} &    f^\da \comp {\et_X}^\da \ar[d]^{f^\da \comp \mbeta_X}  \\
 & & f^\da \comp 1_{\mt X}   }
 \end{gathered}
 \end{equation} 
 commutes. 
 \end{itemize}
\end{definition}

We introduce some terminology. For a relative
pseudomonad $\mt$ over  $J    \co \bC \to \bD$ as in~\cref{thm:relative-pseudomonad},
we refer to the family 
of morphisms $\et_X \co JX \to \mt X$, for $X \in \bC$, as the \emph{unit} of $\mt$, and to the family of
2-cells $\mmu$, $\meta$, and $\mbeta$ as the \emph{associativity}, \emph{right unit}, and \emph{left unit} of $\mt$, respectively. Finally, we refer to the axioms in~\eqref{equ:assoc-relativepseudomonad} and~\eqref{equ:unit-relativepseudomonad} as the~\emph{associativity} and \emph{unit axioms} for $T$, respectively.    
Note that in order to simplify the notation we have omitted the subscripts on the functors $(-)^\da_{X, Y}$ and we will henceforth continue to do so.  We furthermore adopt the convention
of writing~$X$ rather than~$\et_X$ in a subscript of $\mmu$ and $\mbeta$. So, for example, we have
\begin{gather*} 
  \mmu_{f, X}   \co (f^\da \et_X)^\da \to f^\da \comp {\et_X}^\da \, , \\
  \meta_{X}   \co  \et_X \to   {\et_X}^\da \comp \et_X  \, .
  \end{gather*}
We also omit the detailed definition of some 2-cells in diagrams, labelling arrows only with the main 2-cell involved in its definition, and omitting subscripts. In all such cases, the precise definition of the 2-cell can be easily deduced from
its domain and codomain.

We wish to make precise in what sense relative pseudomonads are a generalization of no-iteration 
pseudomonads~\cite[Definition~2.1]{MarmolejoF:noip}. This will be useful in order to relate relative
pseudomonads and ordinary pseudomonads. 

\begin{lemma} \label{thm:cohkle}
Let $\mt$ be a relative pseudomonad over $J \co \bC \to \bD$.
\begin{enumerate}[(i)]  
\item For every $f \co JX \to \mt  Y$ and $g \co JY \to \mt  Z$, the diagram
\begin{equation*}
\begin{gathered}
 \xymatrix@C=1.4cm{
g^\da \comp f \ar[r]^-{\meta_{g^\da \comp f}} \ar@/_1pc/[dr]_-{g^\da \comp \meta_f}  & 
(g^\da \comp f)^\da \comp \et_X \ar[d]^-{\mmu_{g, f}}  \\ 
 & g^\da \comp f^\da \comp \et_X } 
 \end{gathered}
 \end{equation*} 
 commutes.
 \item For every $f \co JX \to \mt  Y$, the diagram
 \[
  \xymatrix@C=1.4cm{
 ({\et_Y}^\da \comp f)^\da \ar[r]^{\mmu_{Y, f}} \ar@/_1pc/[dr]_{(\mbeta_Y \comp f)^\da} & {\et_Y}^\da \comp f^\da 
  \ar[d]^{\mbeta_Y \comp f^\da} \\
  & f^\da } 
  \]
  commutes. 
  \item For every $X \in \bC$, the diagram
  \[
   \xymatrix@C=1.4cm{
   \et_X \ar[r]^-{\meta_{X}} \ar@/_1pc/[dr]_{1} & {\et_X}^\da \comp \et_X \ar[d]^{\mbeta_X \comp \et_X} \\
    & \et_X }
\]
commutes.
\end{enumerate}
\end{lemma}

\begin{proof} The proof is  a modified version of the proof of the
redundancy of three axioms in the original definition of a monoidal
category~\cite{KellyG:onmlcc} (see also~\cite{JoyalA:bratc}), which
has a version also for pseudomonads~\cite[Proposition~8.1]{MarmolejoF:docwsf}.
\end{proof}

\begin{proposition} \label{thm:no-iteration-is-rel-id} 
A no-iteration pseudomonad is the same thing as a relative pseudomonad over the identity.
\end{proposition}

\begin{proof} The two notions involve exactly the same data, except for the direction of
the invertible 2-cells. Then, using the numbering of axioms for a no-iteration pseudomonad 
in~\cite[Definition~2.1]{MarmolejoF:noip}, the equivalence between axioms for a relative pseudomonad
and those for a no-iteration pseudomonad are given as follows: 

\begin{table}[h]
\begin{center}
\begin{tabular}{ccc}
{\em Relative pseudomonads} & & {\em No-iteration pseudomonads} \\[0.5ex] 
Naturality of $\mmu$ & $\Leftrightarrow$ & Axioms 6 and 7 \\ 
Naturality of $\meta$ & $\Leftrightarrow$ & Axiom 4 \\  
Associativity axiom & $\Leftrightarrow$ & Axiom 8 \\ 
Unit  axiom & $\Leftrightarrow$ & Axiom 2 \\ 
\cref{thm:cohkle}, part~(i) & $\Leftrightarrow$ & Axiom 5 \\ 
\cref{thm:cohkle}, part~(ii) & $\Leftrightarrow$ & Axiom 3 \\
\cref{thm:cohkle}, part (iii) & $\Leftrightarrow$ & Axiom 1. 
\end{tabular} 
\end{center}
\end{table} 

\noindent
 Note that it follows that  Axioms 1, 3 and 5 for a no-iteration pseudomonad in~\cite[Definition~2.1]{MarmolejoF:noip} are redundant, in that they can be derived from the others.
\end{proof}

The next remarks use \cref{thm:no-iteration-is-rel-id} and the analysis of the relationship between ordinary pseudomonads and no-iteration pseudomonads in~\cite{MarmolejoF:noip} to show how a pseudomonad can
be regarded as a  relative pseudomonad over the identity $1_\bC \co \bC \to \bC$ and, conversely,
how every relative pseudomonad over the identity determines a pseudomonad.

\begin{remark}[From pseudomonads to relative pseudomonads] \label{thm:pseudo-to-rel}
The combination of \cite[Theorem~6.1]{MarmolejoF:noip} and~\cref{thm:no-iteration-is-rel-id} shows that every pseudomonad  $\mt \co \bC \to \bC$ with data as in~\cref{sec:background} induces a relative pseudomonad over the identity $1_\bC \co \bC \to \bC$. 
Explicitly, for $X \in \bC$, we already
have $\mt  X \in \bC$ and a morphism $\et_X \co X \to \mt  X$ 
as part of the pseudomonad structure. 
For a morphism $f \co X \rightarrow \mt  Y$, 
we define $f^\da  \co \mt  X \rightarrow \mt  Y$ by letting $f^\da \defeq \mut_Y \comp \mt (f)$. 
The three families of invertible 2-cells $\mmu$, $\meta$, $\mbeta$ for a pseudomonad
are then obtained in an evident way. For example, 
for $f \co X \to \mt  Y$, we let $\meta_f \co f \to f^\da \et_X$ be the composite 2-cell 
\[
\xymatrix{
f \ar[r]^-{\lambda} &
 \mut_Y \comp \et_{\mt Y} \comp f \ar[r]^-{\iso} & 
 \mut_Y \comp \mt(f) \comp \et_X  \, .}
 \]
 where the unnnamed isomorphism 2-cell is a pseudonaturality of $\et$. 
\end{remark}

\begin{remark}[From relative pseudomonad over the identity to pseudomonads] \label{thm:rel-to-pseudo}
The combination of ~\cref{thm:no-iteration-is-rel-id} and \cite[Theorem~3.6]{MarmolejoF:noip}  shows that every relative pseudomonad over an identity 
pseudofunctor induces a pseudomonad. The explicit definitions are a bit involved, and therefore checking the coherence diagrams 
directly is not straightforward, but we shall outline a more conceptual account of the 
construction of a pseudomonad from a relative pseudomonad in~\cref{thm:establish-coherence}.
\end{remark}

We introduce a generalization of the notion of pseudoadjunction  between bicategories~\cite{BungeM:cohera,StreetR:fibb},
extending to the 2-categorical setting the notion of a relative adjunction 
considered in~\cite{UlmerF:prodraf} and~\cite[Section~2.2]{AltenkirchT:monnnb}.

\begin{definition} \label{def:frec} 
Let $G \co \bE \to \bD$ be a pseudofunctor. 
A \myemph{relative left pseudoadjoint} $F$ to $G$ over $J \co \bC \to \bD$, 
denoted
\[
\xymatrix{
                       & \bE \ar[d]^{G} \\
\bC \ar[r]_{J}  \ar[ur]^{F} & \bD \, ,}
\]
consists of 
\begin{itemize} 
\item an object $FX \in \bE$, for every object $X \in \bC$; 
\item a family of morphisms $\et_X \co JX \rightarrow GFX$, for $X \in \bC$; 
\item a family of adjoint equivalences
\begin{equation}
\label{equ:frec}
\xymatrix{
    \bD[JX, GA] 
 \ar@<1.2ex>[rr]^(.52){(-)^\sh}
 \ar@{}[rr]|{\bot}    
 & & \ar@<1.2ex>[ll]^{G(-)  \, \et_X} \bE[FX, A] \, , }
\end{equation}
for $X \in \bC$, $A \in \bE$.
\end{itemize} 
\end{definition}

For a relative pseudoadjunction 
as in Definition~\ref{def:frec}, the components of the unit and counit of the adjoint 
equivalences in~\eqref{equ:frec} will be written 
\[
 \meta_f \co f \rightarrow G(f^\sh) \comp \et_X  \, , \quad
 \mepsilon_u \co ( G(u) \comp \et_X)^\sh \rightarrow u \, , 
\]
respectively, where $f \co JX \rightarrow GA$ and $u \co FX \rightarrow A$. 
Note that  a relative pseudoadjunction over the identity $1_\bC \co \bC \to \bC$ is equivalent to a 
pseudoadjunction in the usual sense~\cite{KellyG:eleoot,StreetR:fibb}.
We now establish a relative variant of a standard fact that a pseudoadjunction of bicategories gives rise
to a pseudomonad~\cite{KellyG:eleoot,StreetR:fibb}, namely that 
a relative pseudoadjunction determines a relative pseudomonad. 
The next lemma will be useful. 

\begin{lemma} 
\label{adjfunct} 
Let
\[
\xymatrix{
                       & \bE \ar[d]^{G} \\
\bC \ar[r]_{J}  \ar[ur]^{F} & \bD}
\]
be a relative pseudoadjunction. Then there is an essentially unique way of
extending the function mapping $X \in \bC$ to $FX \in \bE$ to
a pseudofunctor $F \co \bC \rightarrow \bE$ so that the maps
$\et_X \co JX \to GFX$, for $X \in \bC$, become the 1-cell components of
a pseudonatural transformation $\et \co J \Rightarrow GF$.
\end{lemma}

\begin{proof}  For a morphism $f \co X \to Y$ in $\bC$, we define $F(f) \defeq (\et_Y J(f))^\sh$. The unit of the
adjoint equivalence in~\eqref{equ:frec} gives us invertible 2-cells
\begin{equation}
\label{equ:psnatet}
\psi_{f} \co \et_Y \comp J(f) \to GF(f) \comp \et_X 
\end{equation}
for $f \co X \to Y$. For $f \co X \to Y$, $g\co Y \to Z$, we need an invertible 2-cell $\phi_{g, f} \co F(g \comp f) \to F(g) \comp F(f)$.
By the definition, we have
\[
F(g \comp  f)  = (\et_Z \comp J(g \comp f))^\sh \iso  (\et_Z \comp J(g) \comp J(f))^\sh  \, , \quad
F(g) \comp F(f)  = \big( G  \big(   F(g)  \comp F(f)   \big)  \comp \et_X  \big)^\sh \, . 
\]
So we can define $\phi_{g, f}$ using the composite
\[
\xymatrix{
(\et_Z \comp J(g) \comp J(f))^\sh \ar[r]^-{\psi^{-1}} &  
 \big( GF(g) \comp \et_Y \comp J(f) \big)^\sh  \ar[r]^-{\psi}  & 
  \big( GF(g)  \comp GF(f) \comp \et_X \big)^\sh  \ar[r]^-{\iso} &
 \big( G  \big(   F(g)  \comp F(f)   \big)  \comp \et_X  \big)^\sh \, ,}
 \]
 where the unnamed isomorphism is given by the pseudofunctoriality of $G$.
For $X \in \bC$, we need an invertible 2-cell $\phi_X \co F(1_X) \to 1_{FX}$. By definition, 
$F(1_X)  = (\et_X)^\sh = (G(1_{FX}) \comp \et_X)^\sh$, 
and so we define $\phi_X$ to be $\mepsilon_{1_{FX}} \co (G(1_{FX}) \comp \et_X)^\sh \to 1_{FX}$. With these definitions, the pseudonaturality 2-cells for $\et \co J \Rightarrow GF$ are 
the 2-cells in~\eqref{equ:psnatet}. 
The proof of the coherence conditions is routine~(\cf~\cite{KellyG:eleoot,StreetR:fibb}).
\end{proof}

\begin{theorem}  \label{thm:adjkls} 
Let 
\[
\xymatrix{
                   & \bE \ar[d]^G \\
\bC \ar[ur]^F \ar[r]_J & \bD       }
\]
be a relative pseudoadjunction. Then the function sending $X \in \bC$ to $GF(X) \in \bD$ admits the
structure of a relative pseudomonad over $J$.
\end{theorem}

\begin{proof} For  $X \in \bC$, define $\mt  X \defeq GFX$.  The relative pseudoadjunction gives morphisms $\et_X \co X \rightarrow \mt  X$, for
$X \in \bC$, and adjoint equivalences
\begin{equation}
\label{equ:frecc}
\xymatrix{
    \bD[JX, GFY] 
 \ar@<1.2ex>[rr]^(.52){(-)^\sh}
 \ar@{}[rr]|{\bot}    
 & & \ar@<1.2ex>[ll]^{G(-)  \, \et_X} \bE[FX, FY] }
\end{equation}
for $X, Y \in \bC$. We then define  $(-)^\da  \co \bD[JX,\mt  Y] \to \bD[\mt  X,\mt  Y]$
by letting $f^\da \defeq G( f^\sh)$. It now remains to define the families of
invertible 2-cells $\mmu$, $\meta$ and $\mbeta$. 
For~$\mmu_{g, f} \co (g^\da \comp f)^\da \rightarrow g^\da \comp f^\da$, observe that
\[
(g^\da \comp f)^\da = G (G(g^\sh) \comp f)^\sh \, , \quad 
g^\da \comp f^\da  = G(g^\sh) \comp G(f^\sh) \, ,
\]
and so we define $\mmu_{g,f}$ to be the composite 
\[
\xymatrix{
 G\big( G(g^\sh) \comp f \big)^\sh \ar[r]^-{\eta} &
 G\big( G(g^\sh) \comp G(f^\sh) \comp \et_X \big)^\sh \ar[r]^-{\iso} &
    G( G(g^\sh \comp f^\sh) \comp \et_X)^\sh \ar[r]^-{\varepsilon} &
   G(g^\sh \comp f^\sh)  \ar[r]^-{\iso} &  
    G(g^\sh) \comp G(f^\sh) \, .  }
 \]
The 2-cells  $\meta_f \co f \rightarrow f^\da \comp \et_X$ are given by 
the units of the adjunction~(\ref{equ:frecc}), which satisfy the required naturality condition.
For $\mbeta_X \co {\et_X}^\da \rightarrow 1_{\mt X}$, we recall that ${\et_X}^\da = G({\et_X}^\sh)$, and so we define $\mbeta_X$ to be the composite 2-cell
\[
\xymatrix{
 G({\et_X}^\sh) \ar[r]^-{\eta} & 
  G \big((G(1_{FX}) \comp \et_X)^\sh \big)  \ar[r]^-{\varepsilon} &  G(1_{FX}) \ar[r]^{\iso}  & 
   1_{GFX} \, .}
 \]
It remains to establish
the coherence  conditions. While it is possible to show this
directly, it is more illuminating to argue in terms of universal
properties. Simply restating the adjunction in~\eqref{equ:frecc}, we observe that, given
$f \co JX \rightarrow GA$ in $\bD$ and $u \co FX \rightarrow A$ in $\bE$,
for every 2-cell $\phi \co f \rightarrow G(u) \comp \et_X$, there is a 
unique 2-cell $\psi \co f^\sh \rightarrow u$, the adjoint transpose,
such that the diagram 
\[
\xymatrix@C+3ex{
   f \ar[r]^(.4){\meta_f} \ar@/_/[dr]_{\phi} 
& G(f^\sh) \comp \et_X \ar[d]^{G(\psi) \comp \et_X} \\
                             & G(u)   \comp \et_X                  }
\]
commutes. Accordingly, we can characterize  $\mmu_{g,f}$ and $\mbeta_X$  as
follows. There are  2-cells
\[
\tilde{\kappa}_{g, f} \co G( G(g^\sh) \comp f)^\sh \rightarrow G(g^\sh \comp f^\sh) \, , \qquad
\tilde{\kappa}_X \co G({\et_X}^\sh) \rightarrow G(1_{FX}) 
\]
being the image under $G$ of the unique 2-cells  $(G(g^\sh) \comp f)^\sh \rightarrow g^\sh \comp f^\sh$
and ${\et_X}^\sh \rightarrow 1_{FX}$ 
such that the  diagrams 
\[
\xymatrix@C=1.3cm{
  G(g^\sh) \comp f \ar[r]^-{\meta} \ar[d]_{\eta}  &
   G( G(g^\sh) \comp f)^\sh \comp \et_X  \ar[d]^{\tilde{\kappa}_{g,f} \comp \et_X} \\
   G(g^\sh) \comp G(f^\sh) \comp \et_X  \ar[r]_-{\iso}    & G(g^\sh \comp f^\sh) \comp  \et_X          } \qquad 
\xymatrix@C=1.3cm{
   \et_X \ar[r]^-{\meta} \ar[d]_{\iso}  &  
   G({\et_X}^\sh) \comp \et_X \ar[d]^{ \tilde{\kappa}_X \comp \et_X} \\
     1_{GFX} \comp \et_X \ar[r]_-{\iso}   & G(1_{FX}) \et_X    }
\]
commute. The 2-cells $\mmu_{g, f}$ and $\mbeta_X$ then arise by composing these 2-cells with  
pseudofunctoriality 2-cells of $G$. 
The coherence diagrams follow readily, and we give details in the 
2-categorical case, where the characterizing diagrams for $\mmu_{g, f}$
and $\mbeta_X$ reduce to the diagrams
\[
\xymatrix{
   \et_X              \ar[rr]^(.45){\meta} 
                      \ar@/_/[drr]_{1}          &  & 
   {\et_X}^\da \comp \et_X 
                      \ar[d]^{\mbeta_X \comp \et_X}   \\
                                                    &  &  
   \et_X                \, ,  } 
\qquad
\xymatrix{
  g^\da \comp f \ar[rr]^(.45){\meta} \ar@/_/[drr]_(.42){g^\da \comp \meta_f} & &
   (g^\da \comp f)^\da \comp \et_X  \ar[d]^{\mmu_{g, f} \comp \et_X} \\
                             & & g^\da \comp f^\da \comp \et_X  \, .        }
\]
For the associativity condition in~\eqref{equ:assoc-relativepseudomonad},  
we have commuting diagrams
\[
\xymatrix{
(h^\da \comp g)^\da \comp f \ar[r]^(.43){\meta} 
\ar@/_/[ddr]_(.6){(h^\da \comp g)^\da \comp \meta} \ar[ddd]_{\mmu_{h, g} \comp f} &
\big( (h^\da \comp g)^\da \comp f)^\da \comp \et_X
\ar[dd]^{\mmu_{h^\da \comp g, f } \comp \et_X \, ,} \\
  & \\
   &
(h^\da \comp g)^\da \comp f^\da \comp \et_X \ar[ddd]^{\mmu_{h,g} 
\comp f^\da \comp \et_X} \\
h^\da \comp g^\da \comp f \ar@/_/[ddr]_{h^\da \comp g^\da \comp \meta} & 
                              \\
   & \\
  &
h^\da \comp g^\da \comp f^\da \comp \et_X } \qquad \qquad
\xymatrix{
(h^\da \comp g)^\da \comp f \ar[r]^(.43){\meta} 
               \ar[dd]_{\mmu_{g, h} \comp f} & 
\big( (h^\da \comp g)^\da \comp f \big)^\da \comp \et_X
\ar[dd]^{(\mmu_{h,g} \comp f)^\da \comp \et_X} \\
        & \\
h^\da \comp g^\da \comp f \ar[r]^{\meta} \ar@/_1pc/[dr]_(.45){h^\da \comp \meta}
 \ar@/_2pc/[dddr]_{h^\da \comp g^\da \comp \meta} & 
(h^\da \comp g^\da \comp f)^\da \comp \et_X \ar[d]^{\mmu_{h, g^\da \comp f} \comp \et_X} \\
                                           &
h^\da \comp (g^\da \comp f)^\da \comp \et_X
\ar[dd]^{h^\da \comp \mmu_{f,g} \comp \et_X} \\
                                           &
                                           \\
                                           &
h^\da \comp g^\da \comp f^\da \comp \et_X  \, .      }
\]
Note that the triangles in these diagrams commute by part~(i) of~\cref{thm:cohkle}. 
Since both of the composites on the right-hand side of the diagrams
lie in the image of $G$, we deduce by universality that they are equal,
as required. For the unit condition in~\eqref{equ:unit-relativepseudomonad} we have a commuting diagram
\[
\xymatrix{
f 
   \ar[rr]^(.45){\meta_f} \ar[d]_{\meta_f}            &  &      
f^\da \comp \et_X  
   \ar[d]^{{\meta_f}^\da \comp \et_X}                 \\
f^\da \comp \et_X   
   \ar[rr]^(.45){\meta_{ f^\da \comp \et_X}} 
\ar@/_/[drr]_{f^\da \comp \meta_{\et_X}} \ar@/_2pc/[ddrr]_{1} & &
(f^\da \comp \et_X)^\da \comp \et_X 
   \ar[d]^{\mmu_{f, \et_X} \comp \et_X}                  \\
                           & &
f^\da \comp {\et_X}^\da \comp \et_X
   \ar[d]^{f^\da \comp \mbeta_X \comp \et_X}                  \\ 
                            & &
f^\da \comp \et_X  \, .         }
\]
Here, the triangles commute by part~(i) and~(iii) of \cref{thm:cohkle}. 
Again, since the composite of~$(f^\da \comp \mbeta_X ) \comp
 (\mmu_{f, \et_X} ) \comp 
 ({\meta_f}^\da )$
lies in the image of $G$ we deduce by universality
that it equals the identity, as required. \end{proof}

Using~\cref{thm:adjkls}, we can introduce our fundamental example
of a relative pseudomonad, given by the presheaf construction.\footnote{The 
possibility of introducing a notion of relative pseudomonad to include 
 the presheaf construction as an example was mentioned in~\cite[Example~2.7]{AltenkirchT:monnnb}.}

\begin{example} \label{exa:adj-psh} \label{exa:psh}
There is a relative pseudoadjunction of the form
\[
\xymatrix{
                       & \COC \ar[d]^{U} \\
\Cat \ar[r]_-{J}  \ar[ur]^{\pshh} & \CAT \, ,}
\]
where $\COC$ is the 2-category of locally small cocomplete 
categories, cocontinuous functors, and natural transformations,
and  $U \co \COC \rightarrow \CAT$ is the evident forgetful
functor. The category~$\pshX  \defeq [\catX^\op, \Set]$  of presheaves over 
a small category~$\catX$ 
is the colimit completion of~$\catX$ in the sense that, for every locally small
cocomplete~$\catA$, composition with the Yoneda embedding 
$\yo_\catX \co \catX \rightarrow \pshX$ induces an equivalence of categories
\[
\xymatrix@C=1.5cm{
  \CAT[\catX, \catA] &  \ar[l]_-{U(-) \comp \yon_\catX}
  \COC[\pshX, \catA] \, .}
\]
Thus $\pshh$ provides a relative left pseudoadjoint 
to $U$. 
By~\cref{thm:adjkls} we obtain a relative pseudomonad over the inclusion $J \co \Cat \to \CAT$. 
For  a functor $F \co \catX \rightarrow \pshY$, $F^\da \co
\pshX \rightarrow \pshY$ is the  left Kan extension of $F$ along the Yoneda
embedding, defined by the coend formula
\begin{equation}
\label{equ:lan}
F^\da(p)(y) \defeq \int^{x \in \catX} F(x)(y) \times p(x) \, . \
\end{equation}
The invertible 2-cells $\meta_F$ fit into the diagram
\[
\xymatrix{
\catX \ar[r]^-{\yo_X}  \ar@/_1pc/[dr]_{F}  \ar@{}[dr]|{\; \stackrel{\meta_F \; }{\Rightarrow}}
& \pshX \ar[d]^{F^\da} \\
 & \pshY \rlap{.}}
 \]
The 2-cells $\mmu_{F,G}$ and $\mbeta_\catX$ are uniquely determined by
the universal property of left Kan extensions.

There is an analogous relative pseudomonad arising from 
the relative pseudoadjunction
\[
\xymatrix{
                       & \FIL \ar[d]^{U} \\
\Cat \ar[r]_-{J}  \ar[ur]^{\ind} & \CAT \, ,}
\]
where $\FIL$ is the 2-category of locally small
categories with filtered colimits, 
 functors preserving such colimits, and 
all natural transformations~\cite{ArtinM:SGA41}. 
Here, for $\catX \in \Cat$, we define 
$\indX \in \CAT$ to be the full subcategory of $\pshX$ spanned by small filtered colimits
of representables.
 \end{example}

\section{Kleisli bicategories}
\label{sec:kleb}

We  introduce the Kleisli bicategory of a relative pseudomonad, extending to the 2-dimensional setting the 
definition of the Kleisli category of a relative monad~\cite[Section~2.3]{AltenkirchT:monnnb}.

\begin{theorem}\label{thm:kbicat} Let $\mt$ be a relative pseudomonad over $J \co \bC \rightarrow \bD$. Then there is a bicategory~$\kl(\mt )$, called the Kleisli bicategory
of $\mt$, having the objects of $\bC$ as objects, and hom-categories given by 
$\kl(\mt )[X, Y] \defeq \bD[J X, \mt  Y]$,  for $X, Y \in \bC$. 
\end{theorem}

\begin{proof} We begin by defining composition in $\kl(\mt )$. Let $f \co JX \rightarrow \mt  Y$ and $g \co J Y \rightarrow \mt  Z$. We define
$g \kcomp f \co J X \rightarrow \mt Z$ as the composite in $\bD$
\[
\xymatrix{
JX \ar[r]^{f} & \mt Y \ar[r]^{g^\da} & \mt Z \, .}
\]
This obviously extends to 2-cells, so as to obtain the required composition functors. 
For $X \in \bC$, the identity morphism on $X$ in $\kl(\mt)$ is $\et_X \co J X \rightarrow \mt  X$.
For the associativity isomorphisms, let $f \co JX \to \mt Y$, $g \co JY \to \mt Z$ and $h \co
JZ \to \mt V$. Since 
\[
(h \kcomp g) \kcomp f = (h^\da \comp g)^\da f \, , \quad
h \kcomp (g \kcomp f) = h^\da \comp (g^\da \comp f) \, , 
\]
we define the associativity isomorphism $\alpha_{h, g, f} \co (h \kcomp g) \kcomp f  \to  h \kcomp (g \kcomp f)$ to be the composite
2-cell 
\[
\xymatrix@C+2ex{
 (h^\da \comp g)^\da \comp  f \ar[r]^-{\mmu_{h, g} \comp f } &
(h^\da \comp g^\da) \comp f \ar[r]^{\iso}  &
h^\da \comp (g^\da \comp f) \, .}
\]
For the right and left unit, let $f \co JX \to \mt Y$. Since $f \kcomp \et_X = f^\da \comp \et_X$, we
define $\rho_f \co f \kcomp \et_X \to f$ to  be the 2-cell $\meta_f \co f \to f^\da \comp \et_X$. 
Since $\et_Y \kcomp f = {\et_Y}^\da \comp f$, we define $\lambda_f \co \et_Y \kcomp f \to f$ to be the composite 2-cell
\[
 \xymatrix@C+2ex{
{\et_Y}^\da \comp f \ar[r]^{\mbeta_Y \comp f} & 1_{\mt Y} \comp f \ar[r]^{\iso} &  f  \, .}
\] 
We now need to show that these natural isomorphisms satisfy the required coherence conditions. 
We give the proof making explicit the bicategorical structure of $\bC$ and $\bD$. 
We only need to show that the associativity, left unit, and right unit isomorphisms 
satisfy the coherence conditions for a bicategory. 
The coherence axiom for associativity is obtained via the following diagram:
\[
\xymatrix{
    &
    &
    (( k^\da \comp h)^\da \comp g)^\da \comp f \ar[dr]^{\mmu \comp f} \ar[dl]_{(\mmu \comp g)^\da \comp f} & 
    &
    \\
    &
    (k^\da \comp h^\da) \comp g)^\da \comp f \ar[dl]_{\iso} &
    & 
    ((k^\da h)^\da \comp g^\da) \comp f \ar[d]_{(\mmu \comp g^\da) \comp f} \ar[dr]^{\iso} &
    \\
    ( k^\da (h^\da \comp g) )^\da f  \ar[d]_{\mmu \comp f} &
    &
    & 
    ((k^\da \comp h^\da) \comp g^\da) \comp f \ar[dl]_{\iso}  \ar[dr]^{\iso} &
    (k^\da \comp h)^\da (g^\da \comp f) \ar[d]^{\mmu \comp (g^\da \comp f)} \\
    ( k^\da ( h^\da \comp g)^\da)  \comp f \ar[d]_{\iso} \ar[rr]^{(k^\da \comp \mmu) 
\comp f} &
    &
    (k^\da (h^\da \comp g^\da)) \comp f  \ar[d]^{\iso} & 
    &
    (k^\da \comp h^\da) \comp ( g^\da \comp f) \ar[d] \\
    k^\da \comp ( ( h^\da \comp g)^\da \comp f) \ar[rr]_{k^\da \comp (\mmu \comp f)} &
    &
    k^\da ( ( h^\da \comp g^\da) f ) \ar[rr]_{\iso} & 
    &
  k^\da \comp (h^\da ( g^\da \comp f) ) \, ,}
\]
where, starting from the rhombus on the right-hand side and proceeding clockwise, we use naturality 
of the associativity in $\bD$, coherence of associativity in $\bD$, naturality 
of the associativity in $\bD$ again, and finally the associativity coherence axiom for
a relative pseudomonad in~\eqref{equ:assoc-relativepseudomonad}. 
The coherence axiom for the units is obtained via the following diagram:
\[
\xymatrix@R+3ex{
g^\da \comp f \ar[rr]^(.4){{\meta_g}^\da \comp f} \ar@/_/[drrrr]^(.6){\iso}
\ar@/_2pc/[ddrrrrr]_{1} & 
                                                                   &  
(g^\da  \comp \et_Y)^\da \comp f \ar[rr]^{\mmu_{g, Y} \comp f}     & 
                                                                   & 
(g^\da \comp {\et_Y}^\da) \comp f \ar[r]^{\iso} \ar[d]^{(g^\da \comp \mbeta_Y) \comp f} &
g^\da \comp ( {\et_Y}^\da f ) \ar[d]^{g^\da \comp (\mbeta_Y \comp f)} \\
  & 
  &
  & 
  &
(g^\da \comp 1_{\mt Y}) f \ar[r]^{\iso} &
g^\da ( 1_{\mt Y} \comp f) \ar[d]^{\iso}  \\
 & & & & &
g^\da \comp f \, ,}
\]  
where, starting from the triangle on the top left-hand side, we use the coherence axiom for
units of the relative pseudomonad in~\eqref{equ:unit-relativepseudomonad}, naturality of
the associativity in $\bD$, and the coherence axiom for units of $\bD$. 
\end{proof}

Note that, as mentioned in~ \cref{sec:background} for ordinary pseudomonads, $\kl(\mt)$ is only a bicategory even if~$\bC$ and~$\bD$ are 2-categories.

\begin{example}
It is straightforward to identify the bicategory of profunctors of \cref{exa:prof} with the
Kleisli bicategory associated to the relative pseudomonad of presheaves $P$ of Example~\ref{exa:psh}.
First of all, both bicategories have small categories as objects. Secondly, for small categories $\catX$ and $\catY$ we have 
\[
\Prof[\catX, \catY] = [\catY^\op \times \catX, \Set] \, , \qquad
\kl(\pshh)[\catX,\catY] = \CAT[\catX, \pshY] \,. 
\]
Thus, we have a canonical isomorphism of hom-categories
\[
\tau \co \Prof[\catX, \catY]  \to \kl(\pshh)[\catX,\catY]  \, ,
\]
given by exponential adjoint transposition.
Furthermore, these isomorphisms are compatible with composition and identities. For composition,
it suffices to observe that, for profunctors~$F \co \catX \to \catY$ and~$G \co \catY \to \catZ$, there is
a canonical natural isomorphism 
\[
\tau (G \circ F) \iso (\tau G) \circ (\tau F) \, , 
\]
where the composition on the left-hand side is that of $\Prof$, as defined in~\eqref{equ:prof-comp}, while the
composition on the right is the one of $\kl(\pshh)$, which is given by the functorial composite of~$\tau(F)  \co
\catX \to \pshY$ and~$(\tau G)^\da \co \pshY \to \pshZ$, the latter being defined by the formula for
left Kan extensions in~\eqref{equ:lan}. For identities, simply note that, for 
a small category $\catX$, the adjoint
transpose of the identity profunctor on~$\catX$, as defined in~\eqref{equ:prof-id}, is exactly the Yoneda embedding,
which is the identity on~$\catX$ in~$\kl(\pshh)$. 
\end{example}

In  one-dimensional category theory, every monad determines two 
universal adjunctions relating the base category with the category of Eilenberg-Moore
algebras and the Kleisli category for the monad. For pseudomonads, the construction of the bicategory of pseudoalgebras is well-known and it has been considered for no-iteration pseudomonads in~\cite[Section~4]{MarmolejoF:noip}, but we do not
need its counterpart for relative pseudomonads here. We focus instead on the counterpart of the Kleisli adjunction, which
has not been considered for no-iteration pseudomonads. The first step is the following lemma. 

\begin{lemma} \label{thm:kleislihom}
Let  $\mt$ be a relative pseudomonad over $J \co \bC \to \bD$. 
Then the function sending $X \in \kl(\mt)$ to $\mt X \in \bD$ 
admits the structure of a pseudofunctor $G^T \co \kl(\mt ) \to \bD$.
\end{lemma} 

\begin{proof} For $X, Y \in \kl(\mt)$, we define the functor $G^T_{X, Y} \co \kl(\mt )[X,Y] \to  \bD[G^T X,G^T Y]$
to be 
\[
 (-)^\da \co \bD[J X,\mt  Y] \longrightarrow \bD[\mt  X,\mt  Y] \, .
\]
By inspection of the definitions, we can define the pseudofunctoriality $2$-cells of $G^\mt$
to be exactly some of the 2-cells  that are part of the data of a relative pseudomonad, namely
\[
\mmu_{g, f} \co G^T(g \circ f) \to G^T(g) \comp G^T(f) \, , \quad
\mbeta_X \co G^T(\et_X) \to 1_{G^T X} \, .
\]
In order to have a pseudofunctor, we need to verify that the following three coherence diagrams
commute:
\[
\xymatrix{
 						& G^T \big( (h \circ g) \circ f \big) 
						\ar[dl]_{G^T(\mmu_{f,g,h})} \ar[dr]^{\mmu_{h  \circ g, f}}	& \\
 G^T \big( h \circ (g \circ f) \big) 	\ar[d]_{\mmu_{h, g \circ f}} & 								  	& G^T( h \circ g) \comp G^T(f) \ar[d]^{\mmu_{h, g} \comp G^T(f)} \\
 G^T(h) \comp G^T(g \circ f) 		\ar[rr]_{G^T(h) \comp \mmu_{g,f}} & 									& G^T(h) \comp G^T(g) \comp G^T(f) \, ,
 }
 \]
\[
\xymatrix@C=1.4cm{
G^T(f) \ar[r]^-{G^T(\rho_f)} \ar@/_1pc/[drr]_-{1_{G^T(f)}}  & G^T(f \circ \et_X) \ar[r]^-{\mmu_{f, X}} & G^T(f) \comp G^T(1_X) \ar[d]^{G^T(f) \comp \mbeta_X} \\
 & &  G^T(f) \, ,}
\]
\bigskip
\[
\xymatrix@C=4cm{
G^T(\et_Y \kcomp f)   \ar[r]^{\mmu_{Y, f}}   \ar@/_1pc/[dr]_{G^T(\lambda_f)} & 
G^T(\et_Y) \comp G^T(f)  \ar[d]^{\mbeta_Y \comp G^T(f)}  \\
 &  G^T(f) \, . }
\]
The first and second diagrams follow at once from the coherence conditions in~\eqref{equ:assoc-relativepseudomonad} and~\eqref{equ:unit-relativepseudomonad} that are part of the 
definition of a 
relative pseudomonad. The third is part~(ii) of Lemma~\ref{thm:cohkle}.
\end{proof}

By analogy with the one-dimensional case, we expect that
the pseudofunctor $G^T \colon \kl(\mt) \to \bD$ has some form of left pseudoadjoint.
The next result makes this precise.

\begin{theorem} \label{thm:psekle} 
Let $\mt   \co \bC \rightarrow \bD$ be a relative  
pseudomonad over $J \co \bC \to \bD$. Then $G^T \co \kl(\mt ) \rightarrow \bD$  has 
a relative left pseudoadjoint over $J \co \bC \to \bD$, 
\[
\xymatrix{
                       & \kl(\mt) \ar[d]^{G^T} \\
\bC \ar[r]_J \ar[ur]^(.45){F^T} & \bD \, .}
\] 
\end{theorem}

\begin{proof} For $X \in \bC$, 
we  define $F^T X \defeq X$.
Then we have $G^T F^T X = \mt  X$, so the relative pseudomonad 
provides a morphism $\et_X \co J X \rightarrow G^T F^T X$.
For these to act
like the components of the unit of a relative pseudoadjunction
one needs to show that the  functor 
\begin{equation*}
\xymatrix{
  \kl(\mt )[F^T X, Y] 
 \ar[rr]^{G^T(-)  \comp \et_X} & & \bD [JX, G^T Y] }
\end{equation*}
is an adjoint equivalence. Indeed,  
$\kl(\mt )[F^T X, Y] = \bD[JX, \mt  Y] = \bD[JX,G^T Y]$ and the functor~$G^T(-) \comp \et_X$ is naturally isomorphic  to the identity.
This is because, for $f \co JX \to \mt Y$, we have~$G^T(f)  \kcomp \et_X = f^\da \comp \et_X$ and there is an invertible 2-cell~$\eta_f \co f \to f^\da \comp \et_X$, suitably natural.
\end{proof}

\begin{remark} \label{thm:establish-coherence} 
 \cref{thm:psekle} allows us to give a more conceptual
account of the construction of the pseudomonad
associated to a relative pseudomonad over the identity $1_\bC \co \bC \to \bC$ in \cref{thm:rel-to-pseudo}. 
Given a relative pseudomonad $\mt$ over  $1_\bC \co \bC \to \bC$, we can
construct a pseudoadjunction between~$\bC$ and~$\kl(\mt )$
as in \cref{thm:psekle}. Then, the pseudomonad associated to this pseudoadjunction
 is exactly the pseudomonad described in~\cref{thm:rel-to-pseudo}. 
 So we have established again the coherence conditions for a pseudomonad,
 in a clean (albeit indirect) way.

 Let us also note that if we start with a relative pseudomonad $\mt$ over $J  \co \bC \rightarrow \bD$, form the associated Kleisli
relative pseudoadjunction (as in~\cref{thm:psekle}), and take the induced relative pseudomonad (as in~\cref{thm:adjkls}),  we then retrieve the original relative pseudomonad. 
The only issues arise at the 2-cell level and we leave the verifications to the interested readers.
In the other direction, suppose that we start with a relative pseudoadjunction 
\begin{equation*}
\begin{gathered}
\xymatrix{
                       & \bE \ar[d]^{G} \\
\bC \ar[r]_-{J} \ar[ur]^{F} & \bD \, ,}
\end{gathered} 
\end{equation*}
form the induced relative pseudomonad $\mt = GF$ over $J \co \bC \to \bD$ (as in~\cref{thm:adjkls}), and then take the induced relative pseudoadjunction
\begin{equation*}
\begin{gathered}
\xymatrix{
                       & \kl(\mt ) \ar[d]^{G^T} \\
\bC \ar[r]_J \ar[ur]^{F^T} & \bD \, ,} 
\end{gathered}
\end{equation*}
as in~\cref{thm:psekle}.
We expect a
comparison and indeed we have a pseudofunctor  $C \co \kl(\mt ) \to \bE$ defined on objects
by letting~$C(X) \defeq FX$, for $X \in \bC$. On hom-categories,  for $X, Y \in \bC$, we define
\[
C_{X, Y} \co  \bD[JX,\mt  Y]  \to  \bE[FX,FY] 
\]
by letting $C(f) \defeq f^\sh$, where we used that $\mt Y  = GF(Y)$.
\end{remark}

One can compose adjunctions between categories and, similarly,  pseudoadjunctions between bicategories. It does 
not make sense to compose relative pseudoadjunctions, but 
 one can form the composite of 
a relative pseudoadjunction and a pseudoadjunction. 

\begin{proposition} \label{thm:compose-adjunctions}
Let 
\begin{equation*}
\begin{gathered}
\xymatrix{
                       & \bE \ar[d]^{G} \\
\bC \ar[r]_J \ar[ur]^{F} & \bD\, ,} 
\end{gathered} \qquad
\begin{gathered}
\xymatrix{
    \bE
 \ar@<1.2ex>[rr]^(.52){F'}
 \ar@{}[rr]|{\bot}    
 & & \ar@<1.2ex>[ll]^{G'} \bE' }
\end{gathered}
\end{equation*}
be a relative pseudoadjunction and a pseudoadjunction, respectively. Then 
there is a relative pseudoadjunction of the form
\[
\xymatrix{
                       & \bE' \ar[d]^{G G'} \\
\bC \ar[r]_J \ar[ur]^{F'  F} & \; \bD \, .}
\] 
\end{proposition}

\begin{proof} The construction is evident and thus omitted.
\end{proof}

We conclude this section by  extending some results on no-iteration
pseudomonads to relative pseudomonads. In the next
proposition, part~(i) generalizes~\cite[Proposition~3.1]{MarmolejoF:noip},
part~(ii) generalizes~\cite[Proposition~3.2]{MarmolejoF:noip},
while part~(iii) does not seem to have been made explicit yet
 for no-iteration pseudomonads. 

\begin{proposition}\label{prop:funct} \label{prop:unit-nat}

Let $\mt$ be a relative pseudomonad over $J    \co \bC \to \bD$. 
\begin{enumerate}[(i)]
\item The function $\mt  \co \bC \to \bD$  admits the structure of a pseudofunctor.
\item The family of morphisms $\et_X \co JX \to \mt  X$, for $X \in \bC$, admits the structure
of a pseudonatural transformation $\et \co J \to \mt $.
\item The family of
functions $(-)^\da \co \bD[JX,\mt  Y] \to \bD[\mt  X,\mt  Y]$
for $X, Y \in \bC$, admits the structure of a pseudonatural transformation.
\end{enumerate}
\end{proposition}

\begin{proof} Parts (i) and (ii) follow from~\cref{adjfunct} and~\cref{thm:psekle} 
via~\cref{thm:establish-coherence}, but we also give explicit proofs. 
We begin from part~(i). For~$f \co X \to Y$, we define
$\mt (f) \co \mt  X \to \mt  Y$  by letting  $\mt  f \defeq (\et_Y J(f))^\da$. 
We then define the pseudofunctoriality 2-cells. First, we 
need invertible 2-cells $\tau_{g, f} \co \mt (g \comp f) \to \mt (g) \comp \mt (f)$, for 
$f \co X \to Y$ and $g \co Y \to Z$. By definition, we have 
\[
\mt  (g \comp f) =  (\et_Z \comp J(g) \comp J(f))^\da \, , \quad
 \mt  (g) \comp \mt  (f) =  (\et_Z \comp J(g))^\da \comp (\et_Y \comp J(f))^\da \, . 
 \]
We  then define $\tau_{g, f}$ as the composite 2-cell 
\[
\xymatrix@C+7ex{
(\et_Z \comp J(g) \comp J(f))^\da  \ar[r]^-{ \eta}   & 
((\et_Z \comp J(g))^\da \comp \et_Y \comp J(f))^\da \ar[r]^{\mu} & 
  (\et_Z \comp J(g))^\da \comp (\et_Y \comp J(f))^\da }
   \]
Secondly, we need invertible 2-cells $\tau_X \co \mt  (1_X) \to 1_{\mt  X}$ for $X \in \bC$. 
But since $\mt  (1_X) = {\et_X}^\da$ by definition, we  let $\tau_X \defeq \mbeta_X$,
the component of the left unit of the relative pseudomonad.  

One should check the three coherence laws for a pseudofunctor. The coherence law for~$\tau_{g, f}$ involves a pasting of
the associativity condition in~\eqref{equ:assoc-relativepseudomonad}, part~(i) of Lemma~\ref{thm:cohkle}
and all the naturality conditions for the families 2-cells of a relative pseudomonad. 
One of the coherence laws for $\tau_X$
comes from the unit condition in~\eqref{equ:unit-relativepseudomonad}, while the other is from part~(ii) 
of Lemma~\ref{thm:cohkle}. 

For part~(ii), the required pseudonaturality $2$-cell for $f\co X \to Y$ fits into the diagram
\[
\xymatrix@C+3em{
JX \ar[d]_{\et_X}  \ar[r]^{J(f)}  \ar@{}[dr]|{\Downarrow \; \bar{i}_f} & J  Y \ar[d]^{\et_Y} \\
\mt X \ar[r]_{\mt(f)} & \mt  Y}
\]
Since $\mt  (f)  = (\et_Y \comp J(f))^\da$, we can simply let
\begin{equation}
\label{equ:barif}
\bar{i}_f \defeq \meta_{\et_Y \comp J(f)} \, .
\end{equation}
 We should check two coherence conditions for pseudonatural transformations. The composition condition involves a pasting of a naturality of~$\meta$ to a diagram coming from part~(i) of Lemma~ \ref{thm:cohkle}.
The identity condition is just part~(iii) of \cref{thm:cohkle} .

Finally, for part~(iii), to see the pseudonaturality in $X$, take $u\co X' \to X$, observe that $f^\da \comp \mt  (u) =  f^\da \comp (\et_X \comp u)^\da$ and 
note the 2-cell
\[
\xymatrix{
(f \comp u)^\da  \ar[r]^-{\eta} & 
  (f^\da \comp \et_X \comp u)^\da  \ar[r]^{\mu} & 
   f^\da \comp (\et_X \comp u)^\da  
  \, . }
 \]
For the pseudonaturality in $Y$, take
$v\co Y \to Y'$, observe that $(\mt  (v) \comp f)^\da  = ( (\et_{Y'} \comp v)^\da \comp f)^\da$
and $\mt  (v) \comp f^\da = (\et_{Y'} \comp v)^\da \comp f^\da$, 
and note the 2-cell
\[
\xymatrix@C+3ex{
 \big( (\et_{Y'} \comp v)^\da \comp f \big)^\da \ar[r]^{\mu} & 
 (\et_{Y'} \comp v)^\da \comp f^\da  \, .}
 \]
There are coherence conditions to check, but they are straightforward.
\end{proof}

\section{Lax idempotent relative pseudomonads} 
\label{sec:lax-idempotency}

We isolate a special class of relative pseudomonads,
which appears to be the appropriate generalization
to our setting of the notion of a lax idempotent 
2-monad, or Kock-Z\"oberlein 2-monad, or KZ-doctrine~\cite{KockA:monwsa,ZoberleinV:dokat}.
An extensive analysis of these 2-monads, with useful
equivalent formulations, was given in the course of a study
of general property-like 2-monads in~\cite{KellyG:prolt}.
For pseudomonads, the more general notion of lax idempotent
pseudomonad on a bicategory was introduced in~\cite{StreetR:fibb}, with yet another
characterisation of the notion, and studied further in~\cite{MarmolejoF:docwsf,MarmolejoF:kaneli}. 

In order to state the definition of a lax idempotent relative pseudomonad, 
it is convenient to use the notion of a left extension in a bicategory~\cite[\S2.2.]{LackS:a2cc}, which we now recall.  We consider a fixed
morphism $i \co X \to X'$ in a bicategory $\bC$.
By definition, a \emph{left extension} of a morphism~$f \co X \to Y$ along $i$ consists of
a morphism $f' \co X' \to Y$ and a 2-cell $\eta \co f \to f' i$ such that composition
with~$\eta$ induces a bijection between 2-cells $f \to g i$ and 2-cells $f' \to g$ 
for every morphism $g \co X' \to Y$. In this case, one says that $\eta$ exhibits $f'$ as the
left extension of $f$ along $i$. 

Let us fix again a pseudofunctor between bicategories $J \co \bC \to \bD$.

\begin{definition} \label{def:laxidem} 
A \myemph{lax idempotent relative pseudomonad} over $J$ is a relative pseudomonad $\mt $ over $J \co \bC \to \bD$ such that the following conditions hold:
\begin{itemize} 
\item for all $f \co JX \to \mt Y$, the 2-cell $\eta_f \co f \to f^\da \comp \et_X$ 
exhibits $f^\da \co \mt X \to \mt Y$ as a left extension of~$f$ along $\et_X$, 
\item  for all $f \co JX \rightarrow \mt  Y$, $g \co JY \rightarrow \mt  Z$, the diagram
\[
\xymatrix@C=1.8cm{
g^\da f \ar[r]^-{\eta_{g^\da f}}  \ar@/_1pc/[dr]_{g^\da \eta_f} 
& (g^\da \comp f)^\da \et_X \ar[d]^-{\mmu_{f,g} \et_X} \\
 & g^\da \comp f^\da \et_X }
 \]
commutes,  
\item for all $X \in \bC$, the diagram
 \[
  \xymatrix@C=1.8cm{
 \et_X \ar[r]^-{\eta_{\et_X}} \ar@/_1pc/[dr]_-{1} & {\et_X}^\da \comp \et_X \ar[d]^-{\theta_X \comp \et_X} \\ 
 & 
   \et_{X}}
  \]
  commutes.
\end{itemize}
\end{definition}

Our next goal is to give alternative characterizations of lax idempotent 
relative pseudomonads, which we will use to discuss further the relative pseudomonad of presheaves.  For this, we formulate the relative version of the notion of a left Kan pseudomonad introduced in~\cite[Definition~3.1]{MarmolejoF:kaneli}.



 



\begin{definition} \label{thm:lax-local-adjunction}
A \emph{relative left Kan pseudomonad} over  $J \co \bC \rightarrow \bD$
consists of: 
\begin{itemize}
\item an object $\mt  X \in \bD$, for every $X \in \bC$, 
\item a morphism $\et_X \co JX \rightarrow \mt  X$ in $\bD$, for every $X \in \bC$,
\item a morphism $f^\da \co \mt X \to \mt Y$,  for every $f \co JX \to \mt Y$,
\item an
invertible 2-cell $\eta_f \co f \to f^\da \comp \et_X$ which exhibits $f^\da$ as the left extension of~$f$ along $\et_X$, for every $f \co JX \to \mt Y$, 
\end{itemize}
such that the following conditions hold:
\begin{itemize}
 \item the 2-cell $g^\da \comp \eta_f \co g^\da f \to g^\da f^\da \et_X$ exhibits $g^\da f^\da$ as the left extension of $g^\da f$ along $\et_X$, 
 for all $f \co J X \to \mt X$, $g \co JY \to \mt Z$, 
\item the identity 2-cell $1 \co \et_X \to  \et_X$
 exhibits $1_{\mt X}$ as a left extension of $\et_X$ along $\et_X$, for all $X \in \bC$.
 \end{itemize}
\end{definition}

 \medskip
  
Let us now assume we have an object $\mt  X \in \bD$ for every $X \in \bC$,
 a morphism $\et_X \co JX \rightarrow \mt  X$ in $\bD$ for every $X \in \bC$,
a morphism $f^\da \co \mt X \to \mt Y$ for every $f \co JX \to \mt Y$ and
  an invertible 2-cell $\eta_f \co f \to f^\da \comp \et_X$
exhibiting $f^\da$ as the left extension of $f$ along $\et_X$ for every $f \co JX \to \mt Y$.
Note that this gives us all the data for a relative left Kan pseudomonad, but does not
require all its axioms. Below, we refer to this data simply as $\mt \co \bC \to \bD$. 
It is evident that, for all $X, Y \in \bC$ we have an adjunction 
\begin{equation}
\label{equ:etaeps}
\xymatrix{
    \bD[JX, \mt  Y] 
 \ar@<1.2ex>[rr]^-{(-)^\da}
 \ar@{}[rr]|-{\bot}    
 & & \ar@<1.2ex>[ll]^{(-)  \, \et_X} \bD[\mt  X, \mt  Y] \, ,}
\end{equation}
whose unit has components the invertible 2-cells $\eta_f \co f \to f^\da \et_X$, for $f \co JX \to \mt Y$,
and whose counit has components 2-cells that will be written  $\varepsilon_u \co (u \comp \et_X)^\ast \to u$, for $u \co \mt  X \rightarrow \mt  Y$.
We can then state our characterizations of lax idempotent relative pseudomonads as follows. 

\begin{theorem}\label{prop:laxid} 
The following conditions are equivalent:
\begin{enumerate}[(i)] 
\item $\mt \co \bC \to \bD$ admits the structure of a lax idempotent relative pseudomonad over  $J \co \bC \rightarrow \bD$.
\item For every $f \co JX \to \mt Y$, the 2-cell $\varepsilon_{f^\da} \co (f^\da \et_X)^\da \to f^\da $ is invertible and there are isomorphisms
\[
\mu_{f,g} \co (g^\da f)^\da \to g^\da f^\da \, , \quad
\theta_X \co (\et_X)^\da \to 1_{\mt X} \, .
\]
\item The bicategory $\bE$ with objects of the form $\mt X$, for $X \in \bC$, and
hom-categories given by defining $\bE[\mt X, \mt Y]$ to be the full subcategory of
$\bD[\mt X, \mt Y]$ spanned by the morphisms $u \co \mt X \to \mt Y$ such that $u \iso f^\da$, for some
$f \co JX \to \mt Y$ in $\bD$, is a sub-bicategory of $\bD$. 
\item There exists a sub-bicategory $\bE$ of $\bD$ such that $f^\da \co 
\mt X \to \mt Y$ is in $\bE$
for all $f \co JX \to \mt Y$ in $\bD$ and $\varepsilon_u \co (u \et_X)^\da \to u$ is invertible for all $u
\co \mt X \to \mt Y$ in $\bE$.
\item $\mt  \co \bC \to \bD$ is a relative left Kan pseudomonad over  $J \co \bC \rightarrow \bD$.
\end{enumerate}
\end{theorem}

\begin{proof} To prove that (i) implies (ii), observe that 
the axioms for a lax idempotent relative pseudomonad
imply that the following diagram commutes:
\[
\xymatrix@C=1.8cm{
f^\da \et_X \ar[r]^-{\eta_{f^\da \et_X}}  \ar@/_1pc/[dr]_{f^\da \eta_{\et_X}}  \ar@/_3pc/[ddr]_-{1}
& (f^\da \comp \et_X)^\da \et_X \ar[d]^-{\mmu_{\et_X,f} \et_X} \\
 & f^\da \comp \et_X^\da \et_X  \ar[d]^-{f^\da \theta_X \comp \et_X} \\ 
 &  f^\da  \comp \et_{X}}
 \]
By one of the triangular laws for the adjunction in~\eqref{equ:etaeps},  $\varepsilon_{f^\da}$ is the composite
\[
\xymatrix@C=1.8cm{
(f^\da \comp \et_X)^\da \ar[r]^{\mu_{f, \et_X}} &
f^\da \comp (\et_X)^\da \ar[r]^{f^\da \theta_X} &
f^\da }
\]
and it is therefore invertible.  For (ii) implies (iii), the given isomorphisms show that $\bE$ as defined is closed under composition and contains identities. The implication from (iii) to (iv) is immediate.
For the implication from (iv) to (v), observe that the following diagram commutes:
\[
\xymatrix@C=1.8cm{
g^\da f \ar[r]^{\eta_{g^\da f}} \ar[d]_{g^\da \eta_f} & 
  (g^\da \comp f)^\da \et_X  \ar[d]^{(g^\da \comp \meta_f)^\da \et_X}  \\
   g^\da f^\da i_X \ar[r]^{\eta_{g^\da f^\da} \et_X} \ar@/_1pc/[dr]_{1}	& 
 (g^\da \comp f^\da \comp \et_X)^\da \et_X \ar[d]^{\mepsilon_{g^\da \comp f^\da} \et_X} \\
 & g^\da \comp f^\da \et_X}
 \]
 This shows that $g^\da \eta_f$ is the composite of $\eta_{g^\da f}$ (which exhibits $(g^\da f)^\da$ as an extension of $g^\da f$ along~$\et_X$) with an invertible 2-cell. It then follows
 that $g^\da \comp \eta_f$ exhibits
 $g^\da f^\da$ as the left extension of $g^\da f$ along $\et_X$, as required. Similarly,
 we have 
 \[
 \xymatrix@C=1.8cm{
 \et_X  \ar@/_1pc/[dr]_{1} \ar[r]^{\eta_{\et_X}}  &   {\et_X}^\da \comp \et_X \ar[d]^{\varepsilon_{1_{\mt X}}} \\
  & \et_X}
   \]
  which implies that the identity $1 \co \et_X \to  \et_X$ exhibits $1_{\mt X}$ as a left extension of $\et_X$ along $\et_X$.  
 
Finally, we show that (v) implies (i). We begin by defining the remaining parts of the data
for a relative pseudomonad, namely the families of invertible 2-cells $\mu_{f,g} \co (g^\da \comp f)^\da \to g^\da \comp f^\da$ and $\theta_X : {\et_X}^\da \to 1_{\mt X}$. Using the universal
property of $\eta_{g^\da f}$, we define $\mu_{f,g}$ to be the unique 2-cell such that
  \[
 \xymatrix@C=1.8cm{
 g^\da \comp f
 	\ar[r]^-{\eta_{g^\da \comp f}}
  	\ar@/_1pc/[dr]_-{g^\da \comp \eta_f}  & 
(g^\da \comp f)^\da \comp \et_X \ar[d]^-{\mu_{f,g} \comp \et_X} \\
    & 
 (g^\da \comp f)^\da \comp \et_X }
\]
Similarly, using the universal property of $\eta_{\et_X}$, we define $\theta_X$ to be the unique 2-cell such that 
  \[
  \xymatrix@C=1.8cm{
 \et_X \ar[r]^-{\eta_X} \ar@/_1pc/[dr]_-{1} & {\et_X}^\da \comp \et_X \ar[d]^-{\theta_X \comp \et_X} \\ 
 & 
  1_{\mt X} \comp \et_{X} \, .}
  \]
It then remains only to check the coherence conditions of~\cref{thm:relative-pseudomonad}.
There are two ways of doing this. The first is by a diagram-chasing arguments using the universal properties defining~$\mu_{f,g}$ and~$\theta_X$.
The second is to express~$\mu_{f,g}$ and $\theta_X$ in terms of the unit and counit of the adjunction. Taking that approach, first observe that the diagram
\begin{equation}
\label{equ:sublemma} 
\begin{gathered}
\xymatrix{
( h^\da \comp u \comp \et_X)^\da 
\ar^{(h^\da \comp \meta_{u \comp \et_X})^\da}[rr] \ar@/_/[drr] & &
( h^\da \comp (u \comp \et_X)^\da \comp \et_X)^\da \ar[rr]^(.55){\mepsilon_{h^\da \comp 
(u \comp \et_X)^\da}} \ar[d]^{(h^\da \comp \mepsilon_u \comp \et_X)^\da} & &
h^\da \comp (u \comp \et_X)^\da \ar[d]^{h^\da \comp \mepsilon_u} \\
 & & (h^\da \comp u \comp \et_X)^\da \ar[rr]_(.55){\mepsilon_{h^\da \comp u}} 
&  & h^\da \comp u }
\end{gathered}
\end{equation}
commutes by a triangle identity and the naturality of $\mepsilon$.  
The associativity coherence condition is then given by the diagram
\[
 \resizebox{10cm}{8cm}{
\begin{xy}
(50,110)*+{((h^\da \comp g)^\da f)^\da }="1";
(100,90)*+{((h^\da \comp g)^\da \comp f^\da \comp \et)^\da}="2";
(100,60)*+{(h^\da \comp g)^\da \comp f^\da}="3";
(100,30)*+{(h^\da \comp g^\da \comp \et)^\da \comp f^\da}="4";
(100,0)*+{h^\da \comp g^\da \comp f^\da}="5";
(50,-15)*+{h^\da \comp ( g^\da \comp f^\da \comp \et)^\da}="6";
(0,0)*+{h^\da (g^\da \comp f)^\da}="7";
(0,30)*+{(h^\da \comp (g^\da \comp f)^\da \et)^\da  }="8";
(0,60)*+{(h^\da \comp g^\da \comp f)^\da}="9";
(0,90)*+{((h^\da \comp g^\da \comp \et)^\da \comp f)^\da}="10";
(50,75)*+{((h^\da g^\da \et)^\da \comp f^\da \comp \et)^\da}="11";
(50,45)*+{(h^\da \comp g^\da \comp f^\da \comp \et)^\da}="12";
(50,15)*+{(h^\da (g^\da \comp f^\da \comp \et)^\da \et)^\da}="13" ;
{\ar^*+{\fra } "1";"2"};
{\ar^*+{\frb } "2";"3"};
{\ar^*+{\frc } "3";"4"};
{\ar^*+{\frd } "4";"5"};
{\ar_*+{\fre } "6";"5"};
{\ar_*+{\frf} "7";"6"};
{\ar_*+{\frg } "8";"7"};
{\ar_*+{\frh } "9";"8"};
{\ar_*+{\fri} "10";"9"};
{\ar_*+{\frj} "1";"10"};
{\ar_*+{\fro} "11";"12"};
{\ar_*+{\frq} "12";"13"};
{\ar^*+{\frmm} "10";"11"};
{\ar_*+{\fru} "2";"11"};
{\ar^(.4)*+{\frn} "11";"4"};
{\ar^(.4)*+{\frp} "12";"5"};
{\ar_*+{\frr} "13";"6"};
{\ar_*+{\frs} "8";"13"};
{\ar_*+{\frt} "9";"12"};
\end{xy}
 }
\]
where, starting from the top in a clockwise direction, we use 
interchange, two naturalities of~$\mepsilon$, the diagram
in~\eqref{equ:sublemma}, a naturality of $\mepsilon$, a naturality of $\meta$,
and finally an interchange again. 
Finally, the unit condition is given by the following diagram:
\[
\xymatrix{ 
f^\da \ar[rr]^{{\meta_f}^\da}  \ar[d]^{{\meta_f}^\da} \ar@/_3pc/_{1}[dd] 
& & (f^\da \comp \et_X)^\da \ar^{(f^\da \comp \meta_X)^\da}[d] \ar_{1}[lld] \\
(f^\da \comp \et_X)^\da  \ar^{\mepsilon_{f^\da}}[d] &  
& (f^\da \comp (\et_X)^\da \comp \et_X)^\da \ar^{(f^\da \comp 
\mepsilon \comp \et_X)^\da}    [ll] \ar^{
\mepsilon_{f^\da \comp {\et_X}^\da}}[d] \\
f^\da & & f^\da \comp {\et_X}^\da \ar^{f^\da \comp 
\mepsilon_{1_{\mt  X}}}[ll] }
\]
where we have two uses of the triangle identities and a naturality
\end{proof}

\begin{example} \label{thm:exa-lax-idemp-psh}
We can apply~\cref{prop:laxid} to 
show that the relative pseudomonad  of presheaves of~\cref{exa:psh}
is lax idempotent. For this, observe that for $\catX, \catY \in \Cat$, 
there is an adjunction of the form
\begin{equation}
\label{equ:cat-adj-coc}
\xymatrix{
    \CAT[\catX, \pshY] 
 \ar@<1.2ex>[rr]^(.52){(-)^\da}
 \ar@{}[rr]|{\bot}    
 & & \ar@<1.2ex>[ll]^{(-)  \, \yo_X} \CAT[\pshX, \pshY] }
\end{equation}
in which the components of the unit are  natural isomorphisms. 
The left adjoints in~\eqref{equ:cat-adj-coc} 
factor through~$\COC$, the sub-2-category of cocomplete locally small
categories and cocontinuous functors, and if $U \co \pshX \to \pshY$ is cocontinuous, 
then $\mepsilon_U$ is an isomorphism. Part (iii) of~\cref{prop:laxid} then applies.
A similar example arises by considering the Ind-completion,
which is also  lax idempotent. This is because the corresponding left adjoint 
factors through $\FIL$ the sub-2-category of Ind-complete
categories and functors preserving filtered colimits; and 
if $U \co \indX \to \indY$ preserves filtered colimits
then $\mepsilon_U$ is an isomorphism.
\end{example}

\begin{remark} \cref{prop:laxid} and~\cite[Theorems~4.1 and~4.2]{MarmolejoF:kaneli} imply that a pseudomonad~$\mt$ on a bicategory $\bC$ is 
 lax idempotent in the usual sense if and 
only if it is lax idempotent as a relative pseudomonad over the identity $1_\bC \co
\bC \to \bC$ in the sense of~\cref{def:laxidem}.
\end{remark}

\section{Liftings, extensions, and compositions}
\label{sec:liftings-extensions-compositions}

We now discuss a general method to extend a 2-monad to the Kleisli bicategory of a relative pseudomonad, which we will apply
in~\cref{sec:applications} to extend several 2-monads from the 2-category~$\Cat$ of small categories 
and functors to the bicategory~$\Prof$ of small categories and profunctors. 

Let us begin by introducing the setting in which we will be working. We fix a
pseudofunctor between 2-categories $J \co \bC \to \bD$, a relative pseudomonad $\mt$ over $J$ with data as in~\cref{thm:relative-pseudomonad}
and a 2-monad $S \co \bD \to \bD$ with data as in~\cref{sec:background}. We assume that 
the 2-monad $\ms$ restricts along~$J$. Explicitly, this means that we have a dotted functor
\[
\xymatrix@C+3em{
\bC \ar[r]^J \ar@{.>}[d]_{S} & \bD \ar[d]^{S} \\
\bC \ar[r]_J & \bD \, ,}
\]
and that, for $X \in \bC$, the components of the multiplication and the unit, written $\mus_X \co \mss X \to \ms X$ and $\es_X \co X \to \ms X$, respectively, 
are in $\bC$. This implies that the pseudofunctor $J \co \bC \to \bD$ can be lifted to pseudofunctors $J \co \psSalg_\bC \to \psSalg_\bD$ and
$J \co \Salg_\bC \to \Salg_\bD$ (the definition of these 2-categories is recalled in~\cref{sec:background}), making the 
following diagram commute:
\begin{equation}
\label{equ:diagram-lift-inclusion}
\begin{gathered} 
\xymatrix@C+3em{
\Salg_\bC \ar[r]^{J} \ar[d]  & \Salg_\bD \ar[d]\\
\psSalg_\bC  \ar[r]^J  \ar[d] & \psSalg_\bD \ar[d]  \\
\bC \ar[r]_{J} & \bD \, ,}
\end{gathered}
\end{equation} 
where the vertical arrows in the top square are inclusions and those in the bottom square
are forgetful 2-functors.
We shall deal with two types of liftings, one involving only strict
algebras (\cref{def:lif}) and another  one involving both strict algebras and pseudoalgebras (\cref{def:pseudo-lift}). We begin by defining the simpler
type of lifting, involving only  strict algebras.

\begin{definition}  \label{def:lif} A \emph{lifting of $T$ to strict algebras for $S$}, denoted
\[
\xymatrix{
\Salg_\bC \ar[rr]^{\liftT} \ar[d]_U & & \Salg_\bD \ar[d]^U  \\
\bC     \ar[rr]_{\mt }                    & & \bD    \, ,                 }
\]
consists of
\begin{itemize} 
\item a strict algebra structure on $\mt  A$, for every $A \in \Salg_\bC$;
\item a pseudomorphism structure on  $f^\da \co \mt  A \to \mt  B$, for every pseudomorphism $f \co JA \to \mt  B$;
\item a pseudomorphism structure on $\et_A \co J A \to \mt  A$, for every $A \in \Salg_\bC$; 
\end{itemize} 
such that 
\begin{itemize}
\item $\mmu_{f, g} : (g^\da \comp f)^\da \rightarrow g^\da \comp f^\da$ is an algebra 2-cell for every pair
of pseudomorphisms $f \co JA \to \mt  B$ and $g \co J B \to \mt C$; 
\item $\meta_f \co f \rightarrow f^\da \comp \et_A$ is an algebra 2-cell for every pseudomorphism $f \co JA \to \mt  B$;
\item $\mbeta_A \co {\et_A}^\da \rightarrow 1_{\mt  A}$ is an algebra 2-cell for $A \in \Salg_\bC$.
\end{itemize}
\end{definition}

Note that a lifting of $T$ to strict algebras gives immediately a relative pseudomonad $\bar{\mt}$ over the pseudofunctor~$J \co
\Salg_\bC \to \Salg_\bD$ such that applying the forgetful 2-functors to the data of $\bar{\mt}$ returns the corresponding
data of $\mt$. We shall give several examples of liftings of the relative monad of presheaves to strict 
algebras for some 2-monads in~\cref{sec:applications}. However,  it is not useful to work with liftings to categories of strict algebras for other 2-monads, since typically for a strict algebra $A$ 
there is no evident structure of strict algebra structure on $\mt A$. In order to address this situation, we introduce the following definition.

\begin{definition}  \label{def:pseudo-lift}
 A \emph{lifting of $\mt$ to pseudoalgebras for $S$}, denoted
\[
\xymatrix{
\Salg_{\bC} \ar[rr]^{\liftT} \ar[d] & & \psSalg_{\bD} \ar[d]  \\
\bC     \ar[rr]_{\mt }                    & & \bD                     }
\]
consists of the following data:
\begin{itemize} 
\item a pseudoalgebra structure on $\mt  A$, for every $A \in \Salg_{\bC}$;
\item a pseudomorphism structure on  $f^\da \co \mt  A \to \mt  B$, for every pseudomorphism $f \co JA \to \mt B$;
\item a pseudomorphism structure on $\et_A \co JA \to \mt  A$, for every  $A \in \Salg_{\bC}$;
\end{itemize} 
such that 
\begin{itemize} 
\item $\mmu_{f, g} \co (g^\da \comp f)^\da \rightarrow  g^\da \comp f^\da$ is an algebra 2-cell for every pair
of pseudomorphisms $f \co JA \to \mt  B$ and $g \co J B \to \mt C$; 
\item $\meta_f \co f \rightarrow f^\da \comp \et_A$ is an algebra 2-cell  for every pseudomorphism $f \co JA \to \mt  B$;
\item $\mbeta_A \co {\et_A}^\da \rightarrow 1_{\mt  A}$ is an algebra 2-cell for every $A \in \Salg_\bC$.
\end{itemize}
\end{definition}

Similarly to what happened for liftings to strict algebras, a lifting of $T$ to pseudoalgebras gives a relative pseudomonad $\bar{\mt}$, but now over the pseudofunctor~$J \co
\Salg_\bC \to \psSalg_\bD$, again suitably related to $\bar{\mt}$ via the appropriate forgetful 2-functors. Note
here that for the inclusion $J \co \Cat \to \CAT$, the corresponding inclusion $J \co \Salg_\Cat \to \psSalg_\CAT$
is not merely about size distinction, but involves both strict algebras and pseudoalgebras. Indeed, the notion of a relative 
pseudomonad was designed to encompass these situations as well. 

Our next goal is to show how a lifting of a relative pseudomonad $T$  gives rise to
a pseudomonad on the Kleisli bicategory of $T$. In the one-dimensional situation, such a step commonly involves
passing via a distributive law~\cite{BeckJ:disl}. In our setting, where we are dealing with both coherence and size issues, 
such an approach would be rather complicated, as one would have to adapt the theory of pseudo-distributive 
laws~\cite{KellyG:cohtla,MarmolejoF:dislp,MarmolejoF:cohplr} to relative pseudomonads. However, it is possible to take a more direct approach.

\begin{theorem}  \label{thm:main}   \label{thm:kleadj} 
Assume that $\mt$ has a lifting to either strict algebras or pseudoalgebras for $S$. Then $\ms$ has an
extension to a pseudomonad $\extS \co \kl(\mt) \to \kl(\mt)$ on the Kleisli bicategory of $T$.
\end{theorem}

\begin{proof} We only deal with the case of a lifting to pseudoalgebras, since the case of lifting to strict algebras
is completely analogous.  First, we consider the relative pseudomonad $\bar{\mt}$ over $J \co
\Salg_\bC \to \psSalg_\bD$ and its Kleisli bicategory $\kl(\bar{\mt})$. The objects of~$\kl(\bar{\mt})$ are strict algebras with underlying object in $\bC$, and its hom-categories
are given by 
\[
\kl(\bar{\mt})[A, B] = \psSalg_\bD[JA, \mt B] \, .
\]
Secondly, we observe that there is a forgetful pseudofunctor $U \co \kl(\bar{\mt}) \to \kl(\mt)$, defined on objects by sending a strict algebra to its underlying object. 
To define the action on hom-categories, let  $A, B \in \Salg_\bC$. Then, the required functor is determined by the diagram
\begin{equation*}
\begin{gathered}
\xymatrix@C+3em{
\kl(\liftT)[A, B] \ar[r]^{U_{A,B}}  \ar@{=}[d] &  \kl(\mt )[A, B] \ar@{=}[d] \\ 
\psSalg[ JA, \mt  B] \ar[r]_{U_{A,B}} & \bD[ JA, \mt  B ] \, .}
\end{gathered}
 \end{equation*} 
We claim that $U$ has a left pseudoadjoint.  
The action of the left pseudoadjoint on objects is defined by sending $X$ to $\ms X$, the free pseudoalgebra on $X$ (which is in fact a strict algebra since $\ms$ is
a 2-monad). 
Next, for~$X \in \bC$, we define morphisms $\exte_X \co X \to \ms X$ in $\kl(\mt )$ as the composite
\begin{equation*}
\xymatrix{
JX \ar[r]^-{e_X} &
\ms  J X = J \ms X \ar[r]^-{\et_{\ms X}} &
\mt  \ms X }
\end{equation*}
 in $\bD$. 
We wish to show that these are suitably universal. For this, let us observe that the diagram 
\begin{equation}
\label{equ:usenow}
\begin{gathered}
\xymatrix@C+2cm{
\kl(\mt )[X, A]  \ar@{=}[dd] & \kl(\liftT)[\ms X, A]  \ar[l]_{U(-) \circ \exte_X} \ar@{=}[d] \\
  & \psSalg_\bD[JSX, \mt A] \ar@{=}[d] \\ 
\bD[JX, \mt A]  & \psSalg_\bD[SJX, \mt A] \ar[l]^{U(-) \comp e_X}   }
\end{gathered}
\end{equation} 
commutes up to natural isomorphism, since if $f \co SJX \to \mt A$ is a pseudomorphism, then 
\[
f \kcomp \exte_X = f^\da \comp \exte_X = f^\da \comp \et_{\ms X} \comp e_X \iso f \comp e_X \, .
\]
Since the horizontal arrow at the bottom of~\eqref{equ:usenow} is an equivalence, we have the desired 
universality of the morphism $\exte_X$. Now that we have a pseudoadjunction 
\[
\xymatrix{
  \kl(\mt )
 \ar@<1.2ex>[rr]^-{F}
 \ar@{}[rr]|{\bot}    
 & & \ar@<1.2ex>[ll]^-{U} 
  \kl(\liftT) \, ,}
\]
we obtain the desired extension $\tilde{S}$ as the pseudomonad associated to this pseudoadjunction.
\end{proof}

We conclude this section by showing how to compose  a relative pseudomonad and a 2-monad.

\begin{theorem} \label{thm:lnl}
Assume that $\mt$ admits a lifting to
pseudoalgebras of $S$. Then the function sending $X \in \bC$ to $\mts(X) \in \bD$ 
admits the structure of a relative pseudomonad over $J \co \bC \to \bD$.
\end{theorem}

\begin{proof} First, recall that, by~\cref{thm:psekle}, we have a relative pseudoadjunction 
\begin{equation}
\label{equ:first-rel-adj}
\begin{gathered}
\xymatrix{
                                                  & \kl(\mt )      \ar[d]^{G^T} \\
\bC \ar[ur]^{F^T}  \ar[r]_J   & \bD    \, .                          }
\end{gathered}
\end{equation} 
Secondly, let us consider the pseudomonad $\extS \co \kl(\mt ) \to \kl(\mt )$ constructed in the proof 
of~\cref{thm:main} and its 
associated Kleisli bicategory $\kl(\extS)$. Applying again~\cref{thm:psekle}, this time in the
case of an ordinary pseudomonad, we have a pseudoadjunction
\begin{equation}
\label{equ:klexts}
\xymatrix{
    \kl(\mt ) 
 \ar@<1.2ex>[rr]^{F^{\extS}} 
 \ar@{}[rr]|{\bot}    
 & & \ar@<1.2ex>[ll]^{G^{\extS}} \kl(\extS) \, . }
\end{equation}
 By~\cref{thm:compose-adjunctions}, we can then compose the pseudoadjunctions in~\eqref{equ:first-rel-adj} and the relative pseudoadjunction in~\eqref{equ:klexts} so as to obtain a new relative pseudoadjunction 
\[
\xymatrix{
                                    & \kl(\extS)     \ar[d]^{G^T G^{\extS}} \\
\bC \ar[ur]^{F^{\extS} F^T}  \ar[r]_J  & \bD    \, .                        }
\]
By~\cref{thm:adjkls} we then obtain a relative pseudomonad over $J \co \bC \to \bD$. 
Unfolding the definitions, one readily checks that
the underlying function of this relative pseudomonad sends $X \in \bC$ to $\mts (X) \in \bD$,
as required.
\end{proof}

\section{Substitution monoidal structures}
\label{sec:applications}

We  apply our results to obtain a homogeneous method for extending several 2-monads from the 2-category $\Cat$ of
small categories and functors to the bicategory $\Prof$ of small categories and profunctors, encompassing all the examples
considered in the theory of variable binding~\cite{FioreM:abssvb,PowerJ:binsgc,TanakaM:abssvllb},  concurrency~\cite{CattaniG:proomb}, species of structures~\cite{FioreM:carcbg}, models of the differential $\lambda$-calculus~\cite{FioreM:matmcc}, and operads~\cite{GambinoN:opebaf}.

The simplest examples of liftings for the relative pseudomonad for presheaves are with respect to 2-monads on $\CAT$
whose strict algebras are locally small categories 
equipped with suitable classes of limits.
These 2-monads are co-lax, as discussed in~\cite{KellyG:prolt}. The specific
examples of 2-monads that we consider here are those for categories with terminal object,
categories with chosen finite products (by which we mean categories with chosen terminal object and binary products) and categories with chosen finite limits (by which we mean categories with chosen terminal object and pullbacks). 
Each of these 2-monads is flexible in the sense of~\cite{BirdG:fleltc,BlackwellR:twodmt}
and restricts along the inclusion $J \co \Cat \to \CAT$
to a 2-monad on the 2-category $\Cat$ of small categories, so as to
determine a situation as in~\eqref{equ:diagram-lift-inclusion}.
We speak of small (or locally small) strict algebras to indicate
small (or locally small) categories equipped with a strict algebra structure. 

We make some preliminary observations about the pseudomorphisms (\cf~\eqref{equ:marcelo})
in the cases under consideration. Since limits are determined 
up to a unique isomorphism, the pseudomorphisms are exactly the functors that preserve
the specified limits in the usual up to isomorphism sense:
the coherence conditions for a pseudomorphism are automatic~\cite{KellyG:prolt}. Similarly,
one sees directly that any 2-cell between functors preserving 
the relevant limits is an algebra 2-cell. In the terminology of~\cite{KellyG:prolt}, these 2-monads $S$
are fully property-like. It follows in particular that~$\Salg_\CAT[\catA, \catB]$ can be regarded as a full subcategory of~$\CAT[\catA, \catB]$. All 
this is in fact an abstract consequence of the fact
that the 2-monads in question are all co-lax. That fact
is evident and the general theory appears in~\cite{KellyG:prolt}.

\begin{theorem} \label{thm:first-lift} 
Let $\ms \co \CAT \to \CAT$ be the $2$-monad for  
categories with terminal object, or categories with finite products, or categories with finite limits. 
Then 
the  relative pseudomonad of presheaves 
$\pshh \co \Cat \to \CAT$ has a lifting 
to strict $S$-algebras, 
\[
\bar{\pshh} \co \Salg(\Cat) \rightarrow \Salg(\CAT) \, .
\]
\end{theorem}

\begin{proof} Let us begin by observing that we have a choice of limits in $\Set$, 
so for any $\catX \in \Cat$, $\pshX$ has
chosen limits defined pointwise. Thus, there is a
strict $S$-algebra structure on $\pshX$. Furthermore, 
the Yoneda embedding $\yo_\catX \co
\catX \rightarrow \pshX$ preserves those limits. Hence, if~$\catA \in \Salg_\Cat$ is a small strict algebra,
then $\yo_\catA \co \catA \rightarrow \pshA$ is a pseudomorphism of
$\ms $-algebras in an evident fashion. Composition with
$\yo_\catA$ thus gives us a functor
\[
\xymatrix{
\Salg_\CAT[\catA, \pshB] & & \Salg_\CAT[\pshA, \pshB] \, . \ar[ll]^{(-)\comp \yo_\catA} }
\]
Now, suppose that $F \co \catA \rightarrow \pshB$ is a pseudomorphism, that is to say,
$F$ preserves the relevant limits.
Then the left Kan extension $F^\da \co \pshA
\rightarrow \pshB$ also preserves these limits. This is 
critical, and for
the separate classes of limits needs to be proved on a case
by case basis. The case when $\ms $ is the 2-monad for a terminal object
is simple. If $F$ preserves the terminal, then so does 
$F^\da \comp \yo_\catA$ (being naturally isomorphic to $F$). But the Yoneda
$\yo_\catA \co \catA \to \pshA$ preserves the terminal object, and hence
so does $F^\da$. 
 The case when $\ms $ is the 2-monad for finite products can be seen as a corollary of the results in~\cite{ImG:unipcm} (see also~\cref{thm:lift}), but we 
 provide a direct argument. Suppose that $F$ and hence
$F^\da \comp \yo_\catA$ preserves finite products. As the Yoneda
$\yo_\catA \co \catA \to \pshA$ preserves finite products, $F^\da$
preserves finite products of representables. But the objects 
of $\pshA$ are colimits of representables. Since 
$F^\da$ and products with objects (are left adjoints
and so) preserve colimits, it follows that $F^\da$
preserves finite products. 
Finally, the case when $S$ is the monad for finite limits
is similar, though in this case the result is standard.
If $\catA$ has finite limits and $F \co \catA \to \pshB$ preserves finite limits,
then $F$ is flat~\cite[\S VII.10, Corollary 3]{MacLaneS:shegl} and hence $F^\da$ preserves finite limits. 

Thus, in each case, $F^\da$ is a pseudomorphism of  strict $S$-algebras;
and, as we observed above, any
2-cell between pseudomorphisms will be an algebra 2-cell.
Hence, the left Kan extension gives us a functor
\begin{equation*}
\begin{gathered}
\xymatrix{
\Salg_\CAT[\catA, \pshB] \ar[rr]^(.45){(-)^\da} 
& & \Salg_\CAT[\pshA, \pshB] \, . }
\end{gathered}
\end{equation*}

Now we exploit the fact that the relative pseudomonad for
presheaves is lax idempotent (see~\cref{thm:exa-lax-idemp-psh}). So we have an adjunction
\begin{equation}
\xymatrix@C+3ex{
    \CAT[\catA, \pshB]
 \ar@<1.2ex>[r]^(.48){(-)^\da}
 \ar@{}[r]|(.48){\bot}    
 & \ar@<1.2ex>[l]^(.52){(-)  \comp \yo_\catA} 
\CAT[\pshA, \pshB ] \, .}
\end{equation}
We observed that $\Salg[\catA, \pshB]$ and $\Salg[\pshA, \pshB]$
are full subcategories of $\CAT[\catA, \pshB]$ and 
$\CAT[\pshA, \pshB]$, and it is clear from the above
discussion that this 
adjunction  restricts to an adjunction
\[
\xymatrix@C+3ex{
    \Salg_\CAT[\catA, \pshB]
 \ar@<1.2ex>[r]^(.48){(-)^\da}
 \ar@{}[r]|(.47){\bot}    
 &  \ar@<1.2ex>[l]^(.52){(-)  \, \yo_\catA} \Salg_\CAT[\pshA, \pshB]  \, .}
\]
In view of~\cref{prop:laxid},  the claim is proved.
\end{proof}

\begin{remark}
With the experience of these examples of liftings, it is easy
to give examples of 2-monads which do not lift as above.
\begin{itemize}
\item Consider the 2-monad for a category with zero object
(\ie~an object which is both terminal and initial). No category of presheaves of sets
over a non-empty category has a zero object. So the 2-monad cannot lift.
The same applies to the monad for direct sums or  biproducts
(in the terminology of~\cite{MacLaneS:catwm}).

\item Consider the 2-monad for a category with initial object.
Given a category $\catA$ with initial object, while the presheaf
category $\pshA$ does indeed have an initial object, the
Yoneda embedding does not preserve it. Hence the 2-monad cannot lift.

\item Consider the 2-monad for a category with equalisers.
Given a category $\catA$ with equalisers, the presheaf
category $\pshA$ also has equalisers, and the
Yoneda embedding $\yo_\catA \co \catA \to \pshA$ preserves them. 
But now suppose that $\catA$ has equalisers and that $F \co \catA \to \Set$
preserves them. It does not follow that 
$F^{\da} \co \pshA \to \Set$ preserves equalisers.
For a counterexample one can obviously just take
$\catA$ to be the fork (\ie~the generic equaliser). Then for
example take $F \co \catA \to \Set$
mapping the parallel pair to the identity and twist on $2$
with equaliser $0$. Because of this failure it
follows that the 2-monad cannot lift.
\end{itemize}
\end{remark}

Next, we consider 2-monads associated with
various notions of monoidal category.  To start with, we consider 2-monads which are flexible in the sense of~\cite{BirdG:fleltc,BlackwellR:twodmt}
and we have again a situation  as in~\eqref{equ:diagram-lift-inclusion}.

\begin{theorem}  \label{thm:lift}
 Let $S \co \CAT \to \CAT$ be the 2-monad for  monoidal categories, or symmetric monoidal categories, or monoidal categories in which the unit is a terminal object, 
or symmetric monoidal categories in which the unit is a terminal object.
The relative pseudomonad of presheaves $\pshh \co \Cat \to \CAT$ has a lifting to strict $S$-algebras, 
\[
\bar{\pshh} \co \Salg_\Cat \rightarrow \Salg_\CAT \, .
\]
\end{theorem}

\begin{proof} 
The base case is that of a monoidal category. We discuss that case
and derive the others. We use the analysis in \cite{ImG:unipcm}
of the univeral property of Day's convolution tensor product~\cite{DayB:clocf}.
We write $\Mon$ (respectively, $\MON$) for the 2-category of 
small (respectively, locally small) monoidal categories,
strong monoidal functors and monoidal natural transformations.
 For cocomplete 
categories $\catA$, $\catB$, a functor 
$F \co \catA \times \catB \rightarrow \catC$ is \myemph{separately cocontinuous} if for every 
$a \in \catA$, $b \in \catB$  both $F(a, -) \co \catB \rightarrow \catC$ and
$F(-, b) \co \catA \rightarrow \catC$
are cocontinous. 
We write $\COC[\catA ,\catB; \catC]$ for the
category of such functors and natural transformations between them.
A cocomplete 
category $\catA$ equipped with a monoidal structure
is \myemph{monoidally cocomplete}
if the tensor product is separately cocontinuous. 
We then have a straightforward 2-category $\MonCoc$ of 
monoidally cocomplete locally small categories,
strong monoidal cocontinuous functors, and monoidal
transformations.

In~\cite{DayB:clocf} Day showed how for any small monoidal 
category~$\catA$, 
the category $\pshA$ of presheaves on $\catA$ can be equipped with 
a  monoidal structure, called
the convolution tensor product, which makes $\pshA$ into a biclosed monoidally
cocomplete category, defined by letting
\[
(F_1 \hat{\otimes} F_2)(a) \defeq
\int^{a_1, a_2 \in \catA} 
F_1(a_1) \times F_2(a_2) \times \catA[a, a_1 \otimes a_2]
\]
for $F_1, F_2 \in \pshA$ and $a \in \catA$.  Furthermore, 
the Yoneda embedding $\yo_\catA \co \catA \to \pshA$ has then
the structure of a strong monoidal functor. For $\catA \in \Mon$ and $\catB \in \MonCoc$,
we have the adjoint equivalence obtained in~\cite{ImG:unipcm}
\begin{equation}
\label{equ:dayuniv}
\xymatrix{
    \MON[\catA,  \catB]
 \ar@<1.2ex>[rr]^(.47){(-)^\da}
 \ar@{}[rr]|{\bot \quad}    
 & & \ar@<1.2ex>[ll]^(.54){(-)  \, \yo_\catA} 
\MONCOC [\shyA, \catB] \, ,}
\end{equation}
as required. In particular, for any $\catA \, , \catB \in \Mon$, 
we have an adjoint equivalence
\[
\xymatrix{
    \tcat{MON}[ \catA, \shyB]
 \ar@<1.2ex>[rr]^(.47){(-)^\da}
 \ar@{}[rr]|{\bot \quad}    
 & & \ar@<1.2ex>[ll]^(.54){(-)  \, \yo_\catA} 
\tcat{MONCOC} [\shyA, \shyB] \, .}
\]
In our
terminology, the adjoint equivalences in~\eqref{equ:dayuniv} 
amount to saying that  we have a relative pseudoadjunction 
\[
\xymatrix{
                       & \MONCOC \ar[d] \\
\tcat{Mon} \ar[r] 
\ar[ur]^{\pshh} & \tcat{MON} \, .}
\]
This  provides exactly a lifting of the relative pseudomonad $\pshh \co \Cat \to \CAT$ to a 
relative pseudomonad $\bar{\pshh} \co \Salg_\Cat \rightarrow \Salg_\CAT$.  All these
considerations extend to symmetric monoidal categories, again by the results in~\cite{DayB:clocf,ImG:unipcm}.
For the 2-monads for monoidal categories with the condition that the unit is terminal, the lift
follows from the above, observing that the unit of the convolution monoidal structure is the
Yoneda embedding of the unit on the base category and so it remains a terminal object.
\end{proof} 

Our final group of examples of a lifting  involve  2-monads on $\CAT$ which are not flexible. In this case, we have 
a lifting to pseudoalgebras in the sense of~\cref{def:pseudo-lift}.

\begin{theorem}  \label{thm:extpse}  Let $\ms \co \Cat \to \Cat$ be the 2-monad for either strict monoidal categories, or symmetric strict monoidal categories,  
or strict monoidal category in which the unit is terminal, or symmetric strict monoidal categories in which the unit
is terminal. Then 
the  relative pseudomonad $\pshh \co \Cat \to \CAT$ has a lifting to pseudo-$S$-algebras,
\[
\bar{\pshh} \co \Salg_{\Cat} \rightarrow \psSalg_{\CAT} \, .
\]
\end{theorem} 

\begin{proof} 
There is a direct and an indirect approach to this.
Directly, one  follows through the arguments of the previous
section making the necessary adjustments. 
Indirectly, observe that in each case $S'$, the flexible 2-monad  associated to $S$, is the 2-monad whose 
strict algebras are categories with unbiased structure (in the sense of~\cite{LeinsterT:higohc})
as in the list in~\cref{thm:lift}.
Now $\Salg$ is a full sub-2-category of $\Spalg \iso \psSalg$.
So the lifting of the relative pseudomonad of presheaves $\pshh \co \Cat \to \CAT$
to $\bar{\pshh} \co \Spalg(\Cat) \rightarrow \Spalg(\CAT)$
restricts to~$\Salg(\Cat) \rightarrow \Spalg(\Cat)$.
\end{proof}

\begin{corollary} \label{thm:cat-to-prof}
All the 2-monads on $\Cat$ listed in Theorems~\ref{thm:first-lift},\ref{thm:lift}, and \ref{thm:extpse}
admit an extension to pseudomonads on $\Prof$.
\end{corollary}

\begin{proof} Immediate consequence of Theorem~\ref{thm:main} and Theorems~\ref{thm:first-lift},~\ref{thm:lift}, and \ref{thm:extpse}.
\end{proof} 

For each of the monads $S \co \Cat \to \Cat$ above, one can consider the Kleisli bicategory associated to the pseudomonad
$\tilde{S} \co \Prof \to \Prof$ determined by~\cref{thm:cat-to-prof}. The composition functors of these Kleisli bicategories can be understood as  generalizations of various kinds of 
substitution monoidal structures~\cite{FioreM:secodsa,FioreM:carcbg,FioreM:abssvb,GambinoN:opebaf,GurskiN:opetpc,KellyG:onojpm,SmirnovV:onccts}, among those giving rise to the notions of a many-sorted Lawvere theory and of a coloured
operad. We conclude the paper by illustrating this idea in the case of coloured operads.

\begin{example}  \label{sec:substitution-monoidal-structures}
As an illustration of the theory developed here, we revisit the construction of the bicategory of generalized species of structures of~\cite{FioreM:carcbg} and relate
more precisely its composition with the substitution monoidal structure for coloured operads~\cite{BaezJ:higda} (see also~\cite{DayB:abssec}). For this, let us begin by recalling the definition of the 2-monad
$\ms \co \Cat \to \Cat$ for symmetric strict monoidal categories. Let $\catX \in \Cat$. For $n \in \mathbb{N}$, define $\ms_n(\catX)$ to be the category having as objects $n$-tuples
$\bar{x} = (x_1, \ldots, x_n)$ of objects $x_i \in \catX$ and as morphisms $(\sigma, \bar{f}) \co \bar{x} \to \bar{x}'$ given by pairs consisting of a permutation $\sigma \in \Sigma_n$ and 
an $n$-tuple of morphisms $f_i \co x_i \to x'_{\sigma(i)}$. We then let
\[
\ms(\catX) \defeq \bigsqcup_{n \in \mathbb{N}} S_n(\catX) \, .
\]
The category $\ms(\catX)$ is equipped with a strict symmetric monoidal structure: the tensor product, written $\bar{x} \oplus \bar{x}'$, is given by concatenation of sequences, and the unit is given by the empty sequence;  the symmetry is given by a permutation of identity maps. This definition can be  extended easily to obtain a 2-functor $\ms \co \Cat \to \Cat$. The multiplication of the monad is given by taking a sequence of sequences and forgetting the bracketing, while the unit has components $\es_\catX \co \catX \to \ms(\catX)$ mapping $x \in \catX$ to the singleton sequence $(x) \in \ms(\catX)$. 
By the theory developed above, and \cref{thm:cat-to-prof} in particular, we obtain a pseudomonad 
\begin{equation}
\label{equ:exts}
\extS \co \Prof \to \Prof \, .
\end{equation} 
For our purposes, it is convenient to describe explicitly the relative pseudomonad  associated to~$\extS$.
Its action on objects is the function mapping $\catX \in \Cat$ to $S(\catX) \in \Cat$.  The component of the unit for $\catX \in \Cat$ is the profunctor $\exte_\catX \co \catX \to \ms(\catX)$ 
corresponding to the functor
\[
\xymatrix@C+2ex{
\catX \ar[r]^{\es_\catX} & \ms(\catX) \ar[r]^{\yon_{\ms(\catX)}} & \mps (\catX) \, .}
\]
The extension functors of the relative pseudomonad have the form 
\[
(-)^\sh \co  \Prof[\catX, \ms(\catY)]  \rightarrow  \Prof[\ms(\catX), \ms(\catY)] \, ,
\] 
where $\catX, \catY$ are small categories. For a functor $F \co \catX \to \mps (\catY)$, we can define the functor $F^\sh \co \ms(\catX) \rightarrow   \mps(\catY)$ recalling that, since $\ms(\catY)$ has a symmetric strict monoidal structure, $\mps(\catY)$ has an unbiased  (in the sense of~\cite{LeinsterT:higohc}) symmetric  monoidal structure. Hence, by the universal property of~$\ms(\catX)$, we have an essentially unique $F^\sh$ fitting into a diagram of the following form:
\[
\xymatrix@C+3ex{
\catX \ar[r]^{\es_{\catX}}  \ar@/_1pc/[dr]_-{F}  \ar@{}[dr]|{\quad \cong}  & \ms(\catX) \ar[d]^-{F^\sh} \\
 & \mps(\catY) \, .} 
 \]
For brevity, we omit the description of the invertible natural transformations
\[
\meta_F \co F \Rightarrow F^\sh \circ \tilde{\es}_\catX \, , \qquad
\mmu_{F,G} \co (G^\sh \circ F)^\sh \Rightarrow G^\sh \circ F^\sh \, , \qquad
\kappa_{\catX} \co (\tilde{\es}_\catX)^\sh \Rightarrow \id_{\catX} \, .
\]
The Kleisli bicategory of $\extS$ is the bicategory $\SProf$ of $S$-profunctors defined in~\cite{GambinoN:opebaf}, which has 
small categories as objects and hom-categories defined by 
\[
\SProf[\catX, \catY] = \Prof[\catX, \ms(\catY)] = \CAT[ \ms(\catY)^\op \times \catX, \Set] \, .
\]
Indeed, one can readily check that composition and identity morphisms of $\SProf$, as defined in~\cite{GambinoN:opebaf}, coincide with those
given by instantiating the general definition of a Kleisli bicategory. Following~\cite{GambinoN:opebaf}, we write $\CatSym$ for the bicategory of 
categorical symmetric sequences, which is defined as the opposite of $\SProf$. Thus, the objects of $\CatSym$ are small categories and its
hom-categories are given by
\[
\CatSym[\catX, \catY] \defeq \SProf[\catY, \catX] = \CAT[\ms(\catX)^\op \times \catY, \Set] \, .
\]
We write $F[\bar{x}; y]$ for the result of applying  $F \co \ms(\catX)^\op \times \catY \to \Set$ to $(\bar{x}, y) \in \ms(\catX)^\op \times \catY$.
Given categorical symmetric sequences $F \co \catX \to \catY$ and $G \co \catY \to \catZ$, \ie~functors $F \co \ms(\catX)^\op \times \catY \to \Set$ and $G \co \ms(\catY)^\op \times \catZ 
\to \Set$, their composite $G \circ F \co
\catX \to \catZ$ in $\CatSym$ is given by considering $F$ and $G$ as $S$-profunctors in the opposite direction, taking their
composition in $\SProf$ using the definition of composition in a Kleisli bicategory, and then regarding the result as a categorical symmetric sequence from $\catX$ to
$\catZ$. Explicitly, one obtains that
\begin{multline}
\label{equ:subst-mon-struct}
(G \circ F)(\bar{x}; z) \defeq \bigsqcup_{m \in \mathbb{N}} \int^{(y_1, \ldots, y_m) \in \ms_m(\catY)} G[y_1, \ldots, y_m; z] \times \\
 \int^{\bar{x}_1 \in \ms(\catX)} \cdots \int^{\bar{x}_m \in \ms(\catX)} F[\bar{x}_1; y_1] \times \ldots \times F[\bar{x}_m; y_m] \times \ms(\catX)[\bar{x}, \bar{x}_1 \oplus \ldots \oplus \bar{x}_m] \, .
 \end{multline} 
Happily, this formula yields the definition of the substitution monoidal structure for coloured operads given in~\cite{BaezJ:higdat} by considering the special case where $\catX$ and $\catY$ are discrete and coincide with a fixed set of colours of the coloured operads under consideration. 

The bicategory $\CatSym$ can be seen to be equivalent to the bicategory of generalized species of structures $\Esp$ previously introduced in~\cite{FioreM:carcbg}. To see this, let us recall the
definition of $\Esp$. For this, observe that the duality pseudofunctor $(-)^\bot \co \Prof \to \Prof$ defined by $\catX^\bot \defeq \catX^\op$ allows us to turn this pseudomonad in~\eqref{equ:exts}
into a pseudocomonad. The bicategory $\Esp$ is then defined as the co-Kleisli bicategory of this pseudocomonad. More explicitly, its objects are small categories and its hom-categories are
given by 
\[
\Esp[\catX, \catY] = \Prof[\ms(\catX), \catY] = \CAT[\catY^\op \times \ms(\catX), \Set] \, .
\]
The bicategory $\Esp$ is then equivalent to $\CatSym$, via the pseudofunctor that sends $\catX$ to $\catX^\op$. Indeed,
\[
\CatSym[\catX, \catY] = \CAT[\ms(\catX^\op) \times \catY, \Set] \iso [\catY \times \ms(\catX)^\op, \Set] \iso \Esp[\catX^\op, \catY^\op] \, .
\]
Furthermore, the composition operation of categorical symmetric sequences defined in~\eqref{equ:subst-mon-struct} corresponds exactly to  
composition of generalized species of structures defined via co-Kleisli composition given in~\cite{FioreM:carcbg}. 
\end{example}

\end{document}